\newcommand*{\ARXIV}{}%

\ifdefined\ARXIV
  \documentclass[11pt]{article}
  \pdfoutput=1
\else
  \documentclass[review,onefignum,onetabnum,dvipsnames]{siamart171218}
\fi

\usepackage{latexsym,amssymb}
\usepackage{amsmath,amsfonts}

\usepackage{graphicx, appendix}
\usepackage{color}
\usepackage{amssymb} \usepackage{latexsym} \usepackage{exscale}
\usepackage{epsfig,epstopdf}
\usepackage{verbatim,algorithm,algorithmic}
\usepackage{multirow,multicol,enumitem}
\setlist{nolistsep}
\usepackage{wrapfig,colonequals}
\usepackage[numbers,sort]{natbib}
\usepackage[T1]{fontenc}
\usepackage{graphicx, wrapfig}
\usepackage{color}
\usepackage{subfig}
\usepackage{url}
\usepackage{mathrsfs,mathtools,graphicx}
\usepackage{xspace}
\usepackage{hyperref}
\usepackage{bbm}

\def\calV {{\mathcal V}}
\def\calN {{\mathcal N}}

\def\calC {{\mathcal C}}

\def\calU {{\mathcal U}}

\DeclareMathOperator*{\tr}{tr_\Omega}
\DeclareMathAlphabet{\mathpzc}{OT1}{pzc}{m}{it}

\newcommand{\R}{\mathbb{R}}
\newcommand{\N}{\mathbb{N}}
\newcommand{\ve}{u}

\newcommand{\C}{\mathcal{C}}

\DeclareMathOperator*{\argmax}{argmax}

\renewcommand{\H}{\mathbb{H}}
\newcommand{\Hs}{\mathbb{H}^s(\Omega)}
\newcommand{\Hsd}{\mathbb{H}^{-s}(\Omega)}
\newcommand{\aEIM}{a^{Q_h}_{{\rm EIM}}}

\newcommand{\HL}{ \mbox{ \raisebox{7.2pt} {\tiny$\circ$} \kern-10.7pt} {H_L^1} }

\usepackage[textsize=small]{todonotes}
\newcommand{\HA}[1]{{\color{blue}~\textsf{#1}}}
\renewcommand{\HA}[1]{#1}

\makeatletter
\g@addto@macro\normalsize{%
  \setlength\abovedisplayskip{5pt}
  \setlength\belowdisplayskip{5pt}
  \setlength\abovedisplayshortskip{3pt}
  \setlength\belowdisplayshortskip{3pt}
}
\makeatother

\ifdefined\ARXIV
  \usepackage{geometry}
  \usepackage{hyperref}
  \newcommand{\email}[1]{\href{mailto:#1}{\texttt{#1}}}
  \usepackage{amsthm}
  \newtheorem{remark}{Remark}
  \newtheorem{theorem}{Theorem}
  \newtheorem{proposition}{Proposition}
\else
  \ifpdf
    \DeclareGraphicsExtensions{.eps,.pdf,.png,.jpg}
  \else
    \DeclareGraphicsExtensions{.eps}
  \fi

  \newsiamremark{remark}{Remark}

  \headers{RBM for fractional Laplace equations}{H. Antil, Y. Chen, and A. Narayan}

  \title{Certified reduced basis methods for fractional Laplace equations via extension\thanks{Submitted to the editors DATE.
  \funding{H.~Antil was partially supported by National Science Foundation grants DMS-1818772 and DMS-1521590. Y.~Chen was partially supported by National Science Foundation grant DMS-1719698. A.~Narayan was partially supported by AFOSR FA9550-15-1-0467.}}}

  \author{Harbir Antil\thanks{Department of Mathematical Sciences, George Mason University, Fairfax, Virginia 
    (\email{hantil@gmu.edu}, \url{http://math.gmu.edu/\~hantil/}).}
  \and Yanlai Chen\thanks{Mathematics Department, University of Massachusetts Dartmouth, North Dartmouth, MA. 
    (\email{ychen@umassd.edu}, \url{http://www.faculty.umassd.edu/yanlai.chen/}).}
  \and Akil Narayan\thanks{Department of Mathematics, and Scientific Computing and Imaging Institute, University of Utah, Salt Lake City, UT. 
    (\email{akil@sci.utah.edu}, \url{http://www.sci.utah.edu/\~akil}).}
  }

  \usepackage{amsopn}

  \ifpdf
  \hypersetup{
    pdftitle={RBM for fractional Laplace equations},
    pdfauthor={H. Antil, Y. Chen, and A. Narayan}
  }
  \fi
\fi

\begin{document}

\ifdefined\ARXIV
  \title{Certified reduced basis methods for fractional Laplace equations via extension}
  \author{Harbir Antil\thanks{Department of Mathematical Sciences, George Mason University, Fairfax, Virginia 
    (\email{hantil@gmu.edu}, \url{http://math.gmu.edu/\~hantil/}). H.~Antil was partially supported by National Science Foundation grants DMS-1818772 and DMS-1521590.} \and
  \and Yanlai Chen\thanks{Mathematics Department, University of Massachusetts Dartmouth, North Dartmouth, MA. 
    (\email{ychen@umassd.edu}, \url{http://www.faculty.umassd.edu/yanlai.chen/}). Y.~Chen was partially supported by National Science Foundation grant DMS-1719698.} \and
  \and Akil Narayan\thanks{Department of Mathematics, and Scientific Computing and Imaging Institute, University of Utah, Salt Lake City, UT. 
    (\email{akil@sci.utah.edu}, \url{http://www.sci.utah.edu/\~akil}). A.~Narayan was partially supported by AFOSR FA9550-15-1-0467.}
  }
\else
\fi

\maketitle

\begin{abstract}
  Fractional Laplace equations are becoming important tools for mathematical modeling and prediction. Recent years have shown much progress in developing accurate and robust algorithms to numerically solve such problems, yet most solvers for fractional problems are computationally expensive. Practitioners are often interested in choosing the fractional exponent of the mathematical model to match experimental and/or observational data; this requires the computational solution to the fractional equation for several values of the both exponent and other parameters that enter the model, which is a computationally expensive many-query problem. To address this difficulty, we present a model order reduction strategy for fractional Laplace problems utilizing the reduced basis method (RBM). Our RBM algorithm for this fractional partial differential equation (PDE) allows us to accomplish significant acceleration compared to a traditional PDE solver while maintaining accuracy. Our numerical results demonstrate this accuracy and efficiency of our RBM algorithm on fractional Laplace problems in two spatial dimensions.
\end{abstract}

\ifdefined\ARXIV
\else

\fi

\section{Introduction}
This paper concerns the development of mathematically rigorous and computationally efficient reduced order modeling for a quintessential nonlocal partial differential equation (PDE): the fractional Laplace equation. Let $\Omega \subset \mathbb{R}^n$ be a bounded domain in $n\ge 1$ dimensions with 
boundary $\partial\Omega$. Given a function $f$ (whose regularity we will specify later), we are interested in the solution $u$ to
\begin{equation}\label{eq:frac_elliptic_pde}
 (-\Delta)^{s} u = f  \quad \mbox{in } \Omega , \quad u = 0  \quad \mbox{on } \partial\Omega, 
\end{equation}
where $(-\Delta)^s$, supplemented with homogeneous Dirichlet boundary conditions, is a fractional Laplace operator, with fractional exponent $s \in (0,1)$. Among the competing mathematical definitions of the fractional Laplacian \cite{antil2017fractional,MR3489634, LCaffarelli_LSilvestre_2007a, MWarma_2015a, MR2238879}, we concentrate on the spectral definition (see Section \ref{ssec:fraclap}).  The mathematical approach we consider in this paper to devise numerical solvers for \eqref{eq:frac_elliptic_pde} is an ``extension'' technique, and directly applies to more generic situations where we replace $(-\Delta)^s$ with $\mathcal{L}^s$ where $\mathcal{L} w = -\mbox{div}(A \nabla w) + c w$ when $c$ is \HA{non-negative} and bounded on $\Omega$, and $A \in L^\infty(\Omega)^{n\times n}$ is symmetric and uniformly positive definite. Extension techniques also apply to much more general operators \cite{kwasnicki_extension_2018}.

Fractional Laplace problems are useful in many contexts, for instance, image denoising \cite{HAntil_SBartels_2017a,GH:14,antil2018sobolev}, phase field models \cite{HAntil_SBartels_2017a, MAinsworth_ZMao_2017a}, electrical signal propagation in cardiac tissue where the occurrence of fractional Laplacian has been experimentally validated  \cite{AOBueno_DKay_VGrau_BRodriquez_KBurrage_2014a}, diffusion of biological species \cite{viswanathan1996levy}. In fact all heat kernels under fairly general assumptions are either equivalent to heat kernels of diffusion (exponential), or heat kernels for $2s$-stable processes (polynomial) \cite{AGrigoryan_TKumagai_2008a}. In particular, \HA{the fractional} Laplacian is a special case of a $2s$-stable process. We also refer to \cite{chen2017heat, MR1974415} for a general description of fractional heat kernels and their relation to stochastic processes.

Several strategies exist for computing solutions to \eqref{eq:frac_elliptic_pde}. We refer to \cite{PRStinga_JLTorrea_2010a} for the so called Stinga-Torrea extension which was originally proposed in $\mathbb{R}^n$ in \cite{SAMolcanov_EOstrovskii_2969a,LCaffarelli_LSilvestre_2007a} and has come to be known as the Caffarelli-Silvestre extension. The idea is to equivalently write \eqref{eq:frac_elliptic_pde} as a ``local'' PDE problem on $\mathcal{C} := \Omega \times (0,\infty)$, which can be solved using standard algorithms. Using this idea, finite element approaches have been developed in \cite{RHNochetto_EOtarola_AJSalgado_2014a, meidner2017hp} by truncating the semi-infinite cylinder $\mathcal{C}$ to a finite cylinder $\mathcal{C}_{y_+}$ for $y_+ > 0$. Such a truncation is justified due to the exponential decay of solution in the extended dimension. It is also possible to circumvent truncation and directly approximate solutions to the local problem on the unbounded domain $\mathcal{C}$ by using a spectral method in the extended direction \cite{ainsworth2017hybrid}. 

An excellent alternative to the extension approach is the Dunford-Taylor integral representation of the inverse of the fractional Laplacian by employing the Balakrishnan formula \cite{KYosida_1995a}. Recently in \cite{ABonito_JEPasciak_2015a} the authors have developed a finite element method to solve \eqref{eq:frac_elliptic_pde} via this strategy. In a nutshell, solving \eqref{eq:frac_elliptic_pde} amounts to solving several (independent) standard Poisson problems. Finally, for completeness, we mention that it is also common \HA{to solve} integral Fractional Laplace problems via direct discretization of the integral kernel, which leads to large, dense matrices that must be manipulated and/or \HA{inverted \cite{delia_fractional_2013,GA:JP15}}.  
 
In summary, a variety of tools and strategies are available to solve \eqref{eq:frac_elliptic_pde}, but each one faces challenges in evaluating the map $s \mapsto u$.  For instance, the extension approach requires us to solve a PDE in an increased spatial dimension; the Dunford-Taylor approach requires multiple PDE solves for each $s \mapsto u$ query; and solving the integral fractional Laplace version of \eqref{eq:frac_elliptic_pde} requires special discretization schemes for singular kernels and results in large, dense system matrices.  In addition, the discretization details for each of these approaches generally depends on the particular value of $s$. In short, the provenance of the difficulty in solving \eqref{eq:frac_elliptic_pde} is not necessarily due to a particular numerical scheme, and instead seems to stem from the fact the $(-\Delta)^s$ is a nonlocal operator. 

Adding to this difficulty is the reality that, in practical scenarios involving modeling and prediction, the value of the fractional exponent $s$ is not known \textit{a priori}. In fact, it is common to infer this value based on available observational or experimental data. For instance, in \cite{AOBueno_DKay_VGrau_BRodriquez_KBurrage_2014a} the order $s$ was determined by comparing the computational results with experimental data. Recently more systematic optimization approaches to determine $s$ have been developed in \cite{sprekels2016new}. In such cases, the map $s \mapsto u$, evaluated by solving \eqref{eq:frac_elliptic_pde}, must be queried many times for several values of $s$. Thus, one is actually interested in the family of solutions $\left\{ u(\cdot; s)\;\; | \;\; s \in (0,1) \right\}$. This is the philosophy we adopt in this paper, and in particular we use the reduced basis method to efficiently compute approximations to this family of solutions. 

The main contributions of this paper are as follows:
\begin{itemize}
  \item Development of a certified reduced basis method (RBM) approach for performing model order reduction on the fractional Laplace problem \eqref{eq:frac_elliptic_pde}. That is, we formulate an algorithm to construct a surrogate $u_N(x)$ that is inexpensive to evaluate for many different values of the fractional order $s$ (and data $f$), where $N$ denotes the number of degrees of freedom used to construct $u_N$.
  \item Empirical evidence to support a hypothesis that the family of solutions to \eqref{eq:frac_elliptic_pde} is compressible, and that this compressibility can be accomplished efficiently with reduced basis algorithms. With RBM, we usually observe that $N \sim 10$ is sufficient to achieve an error tolerance of $10^{-6}$ for $s \in (\epsilon, 1-\epsilon)$ for a universal constant $\epsilon$, and hence the surrogate can attain sufficient error with considerably reduced computational effort. 
  \item Introduction of novel RBM algorithmic strategies to address challenges unique to fractional PDEs. In particular, RBM algorithms require error estimates from PDEs in certain normed spaces, but generally require that the definition of such norms be parameter-independent. We provide analysis specific to the problem \eqref{eq:frac_elliptic_pde} that allows us to devise error estimates that are parameter-independent.
\end{itemize}
In this paper, we choose the spectral definition of the operator $(-\Delta)^s$ and use the extension approach as our underlying (``truth'') solver that is input into an RBM procedure. Our rationale for these choices is flexibility and generality of the approach: Recently in \cite{antil2018sobolev} the authors proposed a (variational) extension problem with $s(x) \in [0,1]$, i.e., a spatially varying $s$ including the extreme values 0 and 1. It is shown that using this approach one can approximate functions with jump discontinuities. In the context of image denoising, this approach produces results \HA{better than} the popular total variation-based strategies. Therefore, solving \eqref{eq:frac_elliptic_pde} with spatially-varying $s$ is an emerging problem of interest. Currently neither the Dunford-Taylor nor the integral fractional Laplacian allows $x$-dependent exponent $s$, but the extension approach does allow this. We do not in this article consider $s$ as a function of $x$, but our use of the extension approach implies that algorithm principles here can be extended to this case in future work.  

The following is an overview of this paper: Section \ref{sec:background} introduces notation and tools that we require for our procedure. The novel theoretical and algorithmic portions of this paper are presented in section \ref{sec:mor}. Numerical results demonstrating the efficiency and accuracy of our model order reduction approach are given in section \ref{sec:results}.

\section{Background}\label{sec:background}
We assume throughout that the domain $\Omega$ is Lipschitz polygonal/polyhedral, and consider a small generalization of the problem \eqref{eq:frac_elliptic_pde}:
\begin{align}\label{eq:ppde}
  (-\Delta)^{s} u(\cdot;\mu) = f(\cdot;\nu)  \quad \mbox{in } \Omega , \quad u = 0  \quad \mbox{on } \partial\Omega ,
\end{align}
where the parameter $\mu \in \R^p$ is defined as
\begin{align}\label{eq:mu-def}
  \mu &\coloneqq (s, \nu) \in D, & D \coloneqq (0,1) \times E \subset \R^p.
\end{align}
The right-hand side function $f(\cdot;\nu)$ is a parameterized function that is given as input, and the full parameter space is $p$-dimensional. For simplicity of notation, we will subsequently write $f(\cdot;\mu)$, even though $f$ depends only on the components $\nu$ of $\mu$. Our goal is to develop an efficient and accurate procedure for accomplishing evaluation of the map $\mu \mapsto u(\cdot, \mu)$ for many values of $\mu$. The remainder of this section reviews three major foundational ingredients we require: the definition of the fractional Laplacian, the solution to \eqref{eq:ppde} via an extension problem, and the reduced basis method.

\subsection{Function spaces}\label{ssec:fun-spaces}
The space $L^2(\Omega)$ is the collection of square-integrable functions $f: \Omega \rightarrow \R$ over $\Omega$. Consider the standard (negative) Laplacian $(-\Delta)$ defined over $\Omega$. Via the spectral theorem, we can identify a countable sequence of $L^2(\Omega)$-orthonormal and complete eigenfunctions, $\varphi_k$, satisfying
\begin{align}\label{eq:eigdef}
  -\Delta \varphi_k &= \lambda_k \varphi_k \quad \mbox{in } \Omega, 
  \quad \varphi_k = 0 \quad \mbox{on } \partial\Omega, & k &\in \N,
\end{align}
where $0 < \lambda_1 \le \lambda_2 \le ... \le \lambda_k \le ...$ are the eigenvalues of $-\Delta$. Therefore, any $u \in L^2(\Omega)$ has the convergent expansion
\begin{align}
\label{eq:specdef}
  u &= \sum_{k=1}^\infty u_k \varphi_k, & u_k &\coloneqq \left\langle u, \varphi_k \right\rangle_{L^2(\Omega)}.
\end{align}
We can use the spectrum of the Laplacian to define fractional Sobolev spaces:
\begin{align*}
  \H^s(\Omega) &= \left\{ u = \sum_{k=1}^\infty u_k \varphi_k \in L^2(\Omega) \; : \;
   \|u\|_{\mathbb{H}^s(\Omega)}^2 := \sum_{k=1}^\infty \lambda_k^s u_k^2 < \infty \right\} .
\end{align*}
With this definition, we have $L^2(\Omega) = \H^0(\Omega)$. We will let $\H^{-s}(\Omega)$ denote the dual space of $\H^s(\Omega)$.

The spaces $\H^s$ are in general distinct from the classical fractional Sobolev spaces $H^s$, defined as
\begin{align*}
  H^s(\Omega) := \left\{ u \in L^2(\Omega) \; : \; \int_\Omega \int_\Omega \frac{|u(x)-u(y)|^2}{|x-y|^{n+2s}} \; dxdy < \infty \right\},
\end{align*}
and endowed with the norm 
\begin{align*}
  \|u\|^2_{H^s(\Omega)} &\coloneqq  \|u\|_{L^2(\Omega)}^2 + |u|^2_{H^s(\Omega)}, & 
  |u|^2_{H^s(\Omega)} &\coloneqq \int_\Omega \int_\Omega \frac{|u(x)-u(y)|^2}{|x-y|^{n+2s}} \; dxdy .
\end{align*}
In order to draw an equivalence between $\H^s(\Omega)$ and $H^s(\Omega)$, we require two additional spaces, the first being the $H^s$-closure of $\mathcal{D}(\Omega)$, the space of infinitely continuously differentiable functions with compact support on $\Omega$,
\begin{align*}
  H^s_0(\Omega) \coloneqq \overline{\mathcal{D}(\Omega)}^{H^s(\Omega)} .
\end{align*}
We also need the Lions-Magenes space \cite{MR2328004}, 
\HA{
\begin{align*}
  H^{\frac12}_{00}(\Omega) &\coloneqq \left\{ u \in H^{\frac12}(\Omega) \; : \; \int_\Omega \frac{u^2(x)}{\mbox{dist}(x,\partial\Omega)}\; dx < \infty \right\}, \\ 
  \|u\|^2_{H^{\frac12}_{00}(\Omega)} &= \|u\|^2_{H^\frac12(\Omega)} + \int_\Omega \frac{u^2(x)}{\mbox{dist}(x,\partial\Omega)}\; dx.
\end{align*}
The spaces $\H^s(\Omega)$ and $H^s(\Omega)$ are connected by the relation:
\begin{align*}
  \mathbb{H}^s(\Omega) = \left\{ \begin{array}{ll} 
                                   H^s(\Omega) = H^s_0(\Omega) & \mbox{if } 0 < s < \frac12 , \\
                                   H^{\frac12}_{00}(\Omega)  & \mbox{if } s = \frac12 , \\
                                   H^s_0(\Omega)             & \mbox{if } \frac12 < s < 1. 
                                 \end{array} \right.
\end{align*}
}

\subsection{The (spectral) fractional Laplacian}\label{ssec:fraclap}
There are different definitions of the operator $(-\Delta)^s$; in this manuscript we are concerned with the \textit{spectral} definition. The spectral definition of $(-\Delta)^s$, for $s \ge 0$, operating on $u \in \mathcal{D}(\Omega)$ is
\begin{align}
\label{eq:frac_spec}
    (-\Delta)^s u \coloneqq \sum_{k=1}^\infty u_k \lambda_k^{s} \varphi_k ,
\end{align}
where $u_k$ and $\lambda_k$ are as in \eqref{eq:eigdef} and \eqref{eq:specdef}. By density, this definition can be extended to any $u \in \mathbb{H}^s(\Omega)$. We refer to \cite{antil2017fractional} for the spectral definition of $(-\Delta)^s$ in case of
nonzero boundary conditions. For our definition of the spectral Laplacian, we have $(-\Delta)^s : \H^s(\Omega) \rightarrow \H^{-s}(\Omega)$.

This spectral definition of the fractional Laplacian is difficult to use as explicitly shown above since \HA{it} requires computation of the spectrum and eigenfunctions of the Laplacian. In particular, one can  equivalently write $(-\Delta)^s$ in an integral form \cite{MR1974415,MR3489634} which 
immediately implies that $(-\Delta)^s$ is a nonlocal operator. 

Our ultimate goal is to develop a model-order reduction approach using the reduced basis method that allows efficient multi-query evaluation of $s \mapsto u$, where $u$ is defined by the solution to \eqref{eq:frac_elliptic_pde}. In RBM algorithms, the particular nature in which $s$ enters the PDE affects the computational efficiency of the approach; in \eqref{eq:frac_elliptic_pde}, the equation does \textit{not} exhibit affine dependence on $s$, which makes application of RBM algorithms challenging. We therefore first reformulate the problem as one where the parameter $s$ appears in a more convenient form.

\subsection{The Dirichlet-to-Neumann map and extension problem}\label{ssec:extension}
A hallmark result allows one to rewrite the nonlocal operator as a local one; the price paid is an increase in spatial dimension. \HA{The} first results in this direction in unbounded domains are \cite{SAMolcanov_EOstrovskii_2969a,LCaffarelli_LSilvestre_2007a}, although many generalizations followed in bounded domains culminating in our version below \cite{PRStinga_JLTorrea_2010a,XCabre_JTan_2010a,ACapella_JDavila_LDupaigne_YSire_2011a}. Consider $\mathcal{C}$, the cylindrical extension of $\Omega$ to $\R^{n+1}$, with lateral boundary $\partial_L \mathcal{C}$,
\begin{align*}
  \mathcal{C} &\coloneqq \Omega \times (0, \infty) \subset \R^{n+1}, & \partial_L \mathcal{C} \coloneqq \left\{ (x,y) \in \partial \Omega \times [0, \infty)  \right\}.
\end{align*}
We will continue to use the notation that $x \in \Omega \subset \R^n$ denotes the spatial variable as in the problem \eqref{eq:ppde}, and $y \in [0, \infty)$ is a scalar coordinate in the extended ($(n+1)$st) dimension. The extended problem on $\mathcal{C}$ that pairs with \eqref{eq:ppde} is the PDE
\begin{subequations}\label{eq:cs-extension}
\begin{align}
  -\mathrm{div} \left( h(y;s) \; \nabla \calU(x,y) \right) &= 0, & (x,y) &\in \mathcal{C} \\
  \calU(x,y) &= 0, & (x,y) &\in \partial_L \mathcal{C} \\\label{eq:cs-extension-neumann}
  \frac{\partial \calU(x,0)}{\partial n^a} &= d_s f(x;\nu), & (x,y) &\in \Omega \times \{ 0\},
\end{align}
\end{subequations}
where \HA{
\begin{align}\label{eq:h-def}
  h(y;s) &\coloneqq y^{a(s)}, & \frac{\partial \calU(x,0)}{\partial n^a} &= -\lim_{y \rightarrow 0^+} y^a \partial_y \calU(x,y),
\end{align}
}
and the coefficients $a$ and $d_s$ explicitly depend on $s$,
\begin{align}\label{eq:ad-def}
  a = a(s) &\coloneqq 1 - 2s \in (-1,1), & d_s &= 2^a \frac{\Gamma(1-s)}{\Gamma(s)},
\end{align}
with $\Gamma(\cdot)$ the Euler Gamma function. In \eqref{eq:cs-extension}, $\calU : \R^{n+1} \rightarrow \R$ is the unknown weak solution which we must compute, and the operators $\mathrm{div}$ and $\nabla$ are with respect to the extended variable $(x,y) \in \R^{n+1}$. The connection between the solutions of \eqref{eq:ppde} and \eqref{eq:cs-extension} is established via the trace of $\calU$ on the cylinder bottom:
\begin{align}\label{eq:extension-solution}
  \calU(x, 0;\mu) = u(x;\mu).
\end{align}
The appropriate function space for $\calU$ is 
\begin{align*}
  \HL(y^a,\calC) &\coloneqq \left\{ w \in H^1(y^a,\calC) : \; w = 0 \mbox{ on } \partial_L\calC \right\}, & 
  \|w\|^2_{H^1(y^a,\calC)} \coloneqq \sum_{|\delta|\le 1} \|D^\delta w\|_{L^2(y^a,\C)}^2.
\end{align*}
where $D^\delta$ denotes a weak derivative of order $\delta \in \N_0^d$, where we have used multi-index notation.

We emphasize that \eqref{eq:cs-extension} is a \textit{local} PDE, involving only local operators. It is, however, an $(n+1)$-dimensional problem with a singular/degenerate diffusion coefficient $y^a$. 
The relation \eqref{eq:extension-solution} provides a strategy for numerical computation of the solution to the original nonlocal problem \eqref{eq:frac_elliptic_pde} on $\Omega \subset \R^n$: First compute the solution to the local problem \eqref{eq:cs-extension} on $\mathcal{C} \subset \R^{n+1}$, and subsequently restrict the computed solution to the base of the cylinder $\mathcal{C}$. More importantly for us, the parameter $s$ appears in \eqref{eq:cs-extension} (through $a$ and $d_s$) in a form that is far more convenient for RBM algorithms, and will therefore allow us to accomplish efficient model order reduction. The next section describes the tool we use for model order reduction: the reduced basis method.

\subsection{The Reduced Basis Method}\label{ssec:rbm}
The reduced basis method is a strategy for model order reduction of parametric partial differential equations. Consider a PDE of the form 
\begin{align}\label{eq:rbm-pde}
  \mathcal{L}(\calU; \mu) = f(\mu) \quad \mbox{in } \mathcal{C} , \quad u = 0  \quad \mbox{on } \partial\Omega.
\end{align}
Here, $\mu \in D \subset \R^p$ is a Euclidean parameter that encodes variability in the problem and its solution. We assume that for each $\mu$, the solution $\calU$ lies in a Hilbert space $H$. In this paper, we have $H = \HL(y^a, \calC)$ and $\calU^\calN$ is the solution to a discretized version of \eqref{eq:cs-extension}. RBM algorithms start by assuming availability of an expensive \textit{truth} solution or approximation, which is typically a discretized solution from a finite-dimensional space. As with many RBM algorithms, we will adopt a finite element approximation for this procedure. Let $X_h$ be an $\calN$-dimensional finite element space, where $h$ is the mesh parameter. Let $\calU^\calN(\cdot;\mu)$ denote a computable function from $X_h$ that approximates $\calU(\cdot;\mu) \in H$. Typical assumptions are that $\calN \gg 1$, and that $\calU^\calN(\cdot;\mu)$ is accurate but expensive to compute for each $\mu$.  We will make concrete choices for $X_h$, $\calN$, and $\calU^\calN$ via a finite element approach later. 

RBM solutions seek to provide efficient and accurate approximations to the family of solutions $\calU^\calN(\cdot;\mu)$, $\mu \in D$. Hence, $\calU^\calN: \mathcal{C} \times D \rightarrow \R$, with $\calU^\calN(\cdot;\mu) \in X_h \subset H$, so we define 
\begin{align*}
  M &\coloneqq \left\{ \calU^\calN(\cdot; \mu) \;\; \big|\;\; \mu \in D \right\} \subset X_h,
\end{align*}
as the manifold of (approximate) solutions to \eqref{eq:rbm-pde}.\footnote{Here we assume scalar-valued solutions, but RBM algorithms apply equally well to systems of PDEs, \textit{mutatis mutandis}.} The operation $\mu \mapsto \calU^\calN(\cdot; \mu)$ involves the one-time solution to \eqref{eq:rbm-pde}, and can be expensive if $\mathcal{L}$ is complicated or if the dimension $\calN$ is large. 

Thus, the cost $K$ (e.g., computational expense) can be large, scaling algebraically with $\calN$. If, for example, $\mu \mapsto \calU^\calN$ requires inversion of a generic dense linear system of size $\calN$, then $K \sim \calN^3$. RBM algorithms compute an approximate solution, $\calU_N : \mathcal{C} \times D \rightarrow \R$, where evaluation $\mu \mapsto \calU_N(\cdot, \mu)$ has cost $k \ll K$. This approximate solution $\calU_N$ has the following form, along with an error certification provided by a tolerance $\epsilon$:
\begin{align}\label{eq:rbm-def-and-certification}
  \calU_N(\cdot,\mu) &= \sum_{n=1}^N \calU^\calN(\cdot; \mu_n) c_n(\mu) \in X_h, &
  \left\| \calU^\calN(\cdot;\mu) - \calU_N(\cdot;\mu) \right\|_{X_h} &< \epsilon .
\end{align}
The positive integer $N$ denotes the number of degrees of freedom in the RBM solution $\calU_N$. Large $N$ gives rise to a more accurate solution (smaller $\epsilon$), but requires a larger computational cost $k$ for evaluation. The computational cost $k$ for evaluation of $\calU_N$ is cubic in $N$, and the optimal behavior of $\epsilon$ as a function of $N$ is governed by the Kolmogorov $N$-width of the manifold $M$ in the Hilbert space $X_h$:
\begin{align*}
  d_N(M) = \inf_{\dim V = N} \sup_{\calU \in M} \inf_{v \in V} \left\| \calU - v \right\|_{X_h}.
\end{align*}
Above, $V$ is an $N$-dimensional subspace of $X_h$. The value of $d_N$ is the optimal (smallest) error committed by a linear $N$-term approximation of $M$. In many cases, the $N$-width of a solution manifold decays exponentially in $N$. We assume that $d_N$ for our fractional Laplace problem decays very quickly (e.g., algebraically with high order, or exponentially) with respect to $N$. Numerical experiments suggest that this is true for some nonlocal problems \cite{witman_reduced-order_2016}. Under appropriate assumptions, RBM algorithms create a surrogate $\calU_N$ satisfying \eqref{eq:rbm-def-and-certification}, with 
\begin{align*}
  \epsilon \lesssim \sqrt{d_{N/2}(M)}
\end{align*}
See \cite{binev_convergence_2011,devore_greedy_2013}.

In computational implementations, RBM algorithms are split into two phases: an ``offline'' phase, where the $\mu_n$ are determined and the \textit{snapshots} $\calU(\cdot; \mu_n)$ are computed and stored, requiring $N$ queries of the original, expensive model $\calU$, at a cost scaling like $K N$. In the ``online'' phase, given some $\mu$, the function $\calU_N$ defined in \eqref{eq:rbm-def-and-certification} by evaluating the coefficients $c_n(\mu)$ at a cost $k$. When multiple evaluations of $\calU$ are required at numerous values of $\mu$, the savings gained in the online stage can substantially outweigh the one-time expensive offline phase investment \cite{Rozza_Huynh_Patera, BinevCohenDahmenDevorePetrovaWojtaszczyk, CHMR_Sisc, HesthavenRozzaStamm2015, HaasdonkOhlberger}.

The dimension $\calN$ of the space $X_h$ and the complexity $N$ of the reduced basis approximation \eqref{eq:rbm-def-and-certification}, satisfy $N \ll \calN$ in practice. The truth approximation is expensive (having cost $K$), and we seek to develop an RBM algorithm that queries the truth solver as few times as possible. The equation \eqref{eq:rbm-pde} should generically satisfy two assumptions so that an RBM strategy can be computationally efficient: (i) Computable \textit{a posteriori} error estimates should exist, enabling efficient and accurate selection of the $\mu_n$ during the offline phase, and (ii) the operator $\mathcal{L}$ should have \textit{affine} dependence on the parameter $\mu$ so that only efficient, $\calN$-independent operations are required during the online phase. We make the latter point more explicit in the next section.

\section{MOR for fractional Laplace problems}\label{sec:mor}

This section proposes and discusses our algorithm for parametric model order reduction of the fractional Laplacian problem \eqref{eq:ppde}. Assume that the data $f$ in \eqref{eq:ppde} is bounded,
\begin{align*}
  \sup_{\mu \in D} \| f(\cdot;\mu)\|_{\H^{-s}(\Omega)} < \infty.
\end{align*}
The form of the PDE \eqref{eq:ppde} cannot easily be used to perform RBM-based model order reduction with respect to the parameter $\mu$; in particular, the operator dependence on $s$ is troublesome. The main hurdle in application of an RBM algorithm to \eqref{eq:ppde} is that the parameter $s$ appears with ``non-affine'' dependence in the operator $(-\Delta)^s$. A straightforward RBM algorithm would require an approximate equality of the form,
\begin{align}\label{eq:fraclap-affine}
  (-\Delta)^s u \approx \sum_{j=1}^Q a_j(\mu) L_j(u),
\end{align}
where $a_j$ are $x$-independent functions, and $L_j$ are $\mu$-independent operators. Such an approximate equality is necessary to formulate efficient algorithms. It is unclear how such an approximation above can be accomplished directly for this operator. 

We can partially address this problem by performing model order reduction not on the original fractional Laplacian problem, but instead on the extended problem outlined in Section \ref{ssec:extension}. This idea is the core of the strategy in this paper. The extension applied to \eqref{eq:ppde} results in the new parametric PDE \eqref{eq:cs-extension}. Unlike \eqref{eq:ppde}, this new parametric PDE \eqref{eq:cs-extension} \textit{can} be tackled with existing RBM approaches after appropriate modifications. A many-query solution of \eqref{eq:frac_elliptic_pde} may become prohibitively expensive when the dimension of the problem before extension is $2$ or $3$, and thus model order reduction techniques are particularly needed here.

The parametric bilinear (weak) form of \eqref{eq:cs-extension} is, given a test function $\phi \in \HL(y^a,\calC)$, find $\calU \in \HL(y^a,\calC)$ such
that 
\begin{subequations}\label{eq:parabilinear_s}
\begin{align}
  a(\calU, \phi; \mu) &= f(\phi;\mu), & & \\\label{eq:a-def}
  a(\calU, \phi; \mu) &\coloneqq \frac{1}{d_s} \int_{\C} h(y;s) \nabla \calU \cdot \nabla \phi, &  f(\phi;\nu) &= \langle f(\cdot;\mu), \tr \phi \rangle_{\Hsd \times \Hs}. 
\end{align}
\end{subequations}
We concentrate on designing efficient and accurate RBM algorithms associated to this bilinear weak form. The particulars of this construction are complicated by the facts that (i) the form \eqref{eq:parabilinear_s} is \textit{not} affine in $\mu = (s,\nu)$, but \textit{can} be approximated in a fashion similar to \eqref{eq:fraclap-affine}, and that (ii) $a(\ve, \phi;\mu)$ induces a natural norm (due to the Poincar\'e inequality) \HA{on $H^1\left(\mathcal{C}, y^{1 - 2s}\right)$}, that is parameter-dependent, i.e., $s$-dependent.

\subsection{Nonaffine-to-affine transformations}\label{ssec:eim}
The main hurdle to applying RBM in a straightforward fashion is that $a(\ve, \phi; \mu)$ does not exhibit affine dependence on the parameter $\mu$. Affine dependence is required for efficient implementation of RBM through an \textit{offline-online} decomposition.

\subsubsection{Empirical interpolation}
One of the standard approaches to address non-affine parametric bilinear forms is to approximate them as a sum of affine bilinear forms. One strategy for accomplishing this approximation is the \textit{Empirical Interpolation Method} (EIM) \cite{Barrault_Nguyen_Maday_Patera}. Applied to the function $h(y; s)$ in \eqref{eq:h-def}, this takes the form
\begin{align}
  h(y; s) &\approx h^{q}_{\mathrm{EIM}}(y;s) \coloneqq \sum_{j=1}^{q} \theta_{j,q}(s) h_j(y), & h_j(y) &\coloneqq h(y; s_j),
\label{eq:h_eim}
\end{align}
where the samples $\{s_1, \dots, s_{q}\}$ and the functions $\{\theta_{j,q}(s)\}_{j = 1}^{q}$ are computable with a cost dependent only on $q$. In practice, we take $q = Q_h$   that is $\mathcal{O}(10)$. Note that forming the approximation above does \textit{not} require any computational or theoretical analysis of the unknown solution $\calU$. The quality of approximation $\approx$ above with respect to the $y$ variable is frequently measured in the $L^\infty\left([0, \infty)\right)$ norm, and can generally be constructed with high precision using a small value of $Q_h$. Specifically, the strategy in \cite{Barrault_Nguyen_Maday_Patera} constructs this approximation via the successive iteration and computation for $q = 1, \ldots, Q_h-1$,
\begin{align*}
  s_{q+1} &= \argmax_{s \in (0,1)} \left\| h^{q}_{\mathrm{EIM}}(\cdot;s) - h(\cdot;s) \right\|_{L^\infty([0, \infty))},
\end{align*}
and the $\theta_{j,q}$ functions are cardinal Lagrange functions satisfying $\theta_{j,q}(s_i) = \delta_{j,i}$. The existence and uniqueness of the functions $\theta_{j,q}$ is guaranteed for all values of $q$ for which the objective under the $\argmax$ achieves a non-zero value. The EIM construction is defined by a function space in which $h$ has membership. As shown above, we are performing an EIM approximation on the function space $L^\infty\left( (0, 1)\;, \; L^\infty\left([0, \infty) \right)\right)$. The particular construction of an EIM approximation in this space is complicated by the singular dependence in the $y$ variable at $y=0, \infty$. To partially address this issue, we perform EIM on a $y$-truncated version \eqref{eq:h_eim}. I.e., we define a new EIM approximation by replacing the enclosing function space for $h$ by a version that truncates the $y$ variable:
\begin{align*}
  h \in L^\infty\left( (0, 1)\;, \; L^\infty\left([0, \infty) \right)\right) \hskip 5pt \longrightarrow \hskip 5pt h \in L^\infty\left( (0, 1)\;, \; L^\infty\left([y_-, y+] \right)\right),
\end{align*}
for small and large values of $y_-$ and $y_+$, respectively, and discretize the domain $[y_-, y+]$. In addition, we treat dependence on the $s$ variable in a piecewise fashion to accommodate the different behavior of $h(y; s)$ for $s \le \frac{1}{2}$, and for $s > \frac{1}{2}$. We first define $D_1 \coloneqq (0, \frac{1}{2}]$ and $D_2 \coloneqq (\frac{1}{2},1)$ and introduce two function spaces on them
\begin{align}
\label{eq:F-def}
  F(D_i) &\coloneqq L^\infty((0,1), L^\infty([y_-^i, y_+])), & h &\in F(D_i) & i \in \{1, 2\},
\end{align}
where $y_-^1$ and $y_-^2$ are different small truncation values for $y$ that we will specify in the numerical results section. Next, we construct $h_{q,i}$ and $\theta_{j,q,i}$ as EIM approximations in the space $F(D_i)$:
\begin{subequations}\label{eq:h-eim-def}
\begin{align}
  h(y;s) &= y^{1-2s} \stackrel{F(D_1)}{\approx} \sum_{j=1}^{Q_{h,1}} \theta_{j,q,1}(s) h_{q,1}(y), & s \in D_1  \\
  y h(y;s) &= y^{2-2s} \stackrel{F(D_2)}{\approx} \sum_{j=1}^{Q_{h,2}} \theta_{j,q,2}(s) h_{q,2}(y), & s \in D_2.
\end{align}
\end{subequations}
The special treatment of multiplication by $y$ on $F(D_2)$ is key as it allows us to perform stable and efficient empirical interpolation on a function, $y h(y;s)$, that is bounded at the origin when $s > \frac{1}{2}$. Finally, we are ready to define the piecewise EIM approximation of $h(y; s)$ by taking \HA{$Q_h = Q_{h,1} + Q_{h,2}$}:
\begin{align}
\label{eq:heim-subintervals}
  h(y;s) &\approx h^{Q_h}_{\mathrm{EIM}}(y;s) \coloneqq \sum_{j=1}^{Q_h} h_q(y) \theta_q(s) \nonumber\\\nonumber
         & = \sum_{j=1}^{Q_{h,1}} h_{q,1}(y) \theta_{j,q,1}(s) \mathbbm{1}_{D_1}(s) + \sum_{j=1}^{Q_{h,2}} h_{q,2}(y) \theta_{j,q,2}(s) \mathbbm{1}_{D_2}(s)\\
&  = \left\{ \begin{array}{ll} 
      \sum_{j=1}^{Q_{h,1}} h_{q,1}(y) \theta_{j,q,1}(s), & s \in D_1, \quad y \in [y_-^1, y_+] \\
      \sum_{j=1}^{Q_{h,2}} y^{-1} h_{q,2}(y) \theta_{j,q,2}(s), & s \in D_2, \quad y \in [y_-^2, y_+],\end{array}\right.
\end{align}
We emphasize that computation and evaluation of \eqref{eq:heim-subintervals} is entirely $\calU^\calN$-independent and requires analysis of only the function $h(y;s)$. 

\subsubsection{A bilinear form with affine parametric dependence}
Having constructed $h^{Q_h}_{EIM}$, we now replace $h$ in \eqref{eq:a-def} by $h^{Q_h}_{EIM}$ and define a new bilinear form as an amendment to \eqref{eq:a-def},
\HA{
\begin{align*}
  a^{Q_h}_{\mathrm{EIM}}(\calU, \phi; \mu) &\coloneqq \frac{1}{d_s} \int_{\C} h^{Q_h}_{EIM}(y;s) \nabla \calU \cdot \nabla \phi, &  f(\phi;\mu) &= \langle f(\cdot;\mu), \tr \phi \rangle_{\Hsd \times \Hs}.
\end{align*}
}
Thus our new bilinear form has the decomposition
\begin{equation}\label{eq:a-eim}
\aEIM (\ve, \phi; \mu)= \sum_{q=1}^{Q_h} \theta_q(s) a_q(\ve, \phi) \quad \mbox{ with } \quad a_q(\ve, \phi) = a(\ve, \phi; s_q).
\end{equation}
Using this bilinear form instead of $a(\cdot,\cdot;\mu)$ in \eqref{eq:parabilinear_s} results in a weak form that is amenable to efficient offline-online decomposition. 

Finally, we justify our truncation to the values $y_{\pm}$ in the $y$ variable. Truncation to a sufficiently large upper bound of $y_+$ has theoretical justification due to exponential decay of $\calU$ in $y$ \cite[Proposition~3.1]{RHNochetto_EOtarola_AJSalgado_2014a}, i.e.,
\begin{align*}
  \|\calU(\cdot,y;\mu)\|_{H^1(y^a,\Omega\times (y_+,\infty))} &\lesssim e^{-\sqrt{\lambda_1}y_{+}/2} \|f\|_{\Hsd}, 
\end{align*}
where $\lambda_1$ is the first eigenvalue of $-\Delta$ as defined in \eqref{eq:eigdef}. The lower truncation at $y = y_-$ is likewise reasonable since numerical approximations of $\calU^\calN$ use \textit{interior} quadrature rules for the variable $y$ to approximate the integral in the bilinear form in \eqref{eq:parabilinear_s}; thus, integration over $[0, y_+]$ does not require evaluation of the integrand at $y=0$. Therefore, the choice of the truncation parameter $y_-$ can be, for example, the value of the smallest node in a quadrature rule used to compute the integrand in the bilinear form.

\subsection{The truth and RBM approximations}\label{ssec:fem}
This section more rigorously defines the RBM truth approximation $\calU^\calN$ that we use, 
and is a description of the approximation in \cite{RHNochetto_EOtarola_AJSalgado_2014a}. We 
replace the unbounded domain $\calC$ by $\calC_{y_+}$ with $y_+ > 0$, which has ``lateral'' boundary $\partial_L\calC_{y_+} = (\partial\Omega\times [0,y_+]) \cup \left(\Omega \times \{y_+\}\right)$. 
Alternatively, one can also use the discretization from \cite{ainsworth2017hybrid}, where no truncation of $\calC$ in the $y$-direction is needed. Let $\mathcal{T}_\Omega= \{K_j\}_{j=1}^{\#\mathcal{T}_\Omega}$ be a conforming and quasi-uniform triangulation of $\Omega$ where $K_j \subset \mathbb{R}^n$ is an element that 
is, for all $j$, isoparametrically equivalent to either the unit cube or to a unit simplex in 
$\mathbb{R}^n$. Under the assumption that $\#\mathcal{T}_\Omega = M^n$ for some $M > 0$, we deduce that the element size satisfies $h \propto M^{-1}$. 

We define $\mathcal{I}_{y_+} = \{I_m\}_{m=0}^{M-1}$, where $I_m = [y_m,y_{m+1}]$, as an anisotropic partition of $[0,y_{+}]$, so that $[0,y_{+}] = \cup_{m=0}^{M-1} I_m$. We define the edges $y_m$ of the mesh as in \cite{RHNochetto_EOtarola_AJSalgado_2014a,meidner2017hp}:
 \begin{equation}\label{eq:grading}
  y_m = \left(\frac{m}{M}\right)^\gamma , \quad 
   m = 0,\dots, M, \quad \gamma > \frac{1}{s} .
 \end{equation}
This requirement would therefore prescribe different meshes for different values of $s$. Since we require an $s$-independent grid in order to easily compute the approximation for $\calU_N$ in \eqref{eq:rbm-def-and-certification}, we will choose a value of $\gamma$ that is $s$-independent. In particular, we choose a small, fixed value of $s$, say $s_\ast$, and choose $\gamma = 1/s_\ast$ to generate the mesh in \eqref{eq:grading}. This is described in more detail in Section \ref{sec:results}.

Having defined our $s$-independent triangulation, we are now ready to define our finite element space. We will construct a triangulation $\mathcal{T}_{y_+}$ of $\calC_{y_+}$ via a tensor product, $\mathcal{T}_{y_+} \coloneqq \mathcal{T}_\Omega \times \mathcal{I}_{y_+}$. Then we define the finite element space $X_h$ as 
 \begin{align*}
  X_h = X_{h_{\mathcal{T}_\Omega}} 
   := \Big\{\mathcal{V} &\in C^0(\overline{\calC_{y_+}}) \; : \; \mathcal{V}|_T \in 
   \mathcal{P}_1(K) \otimes \mathbb{P}_1(I) , \\ 
   \quad &\forall T = K \times I \in \mathcal{T}_{y_+} \mbox{ and }
   \mathcal{V}|_{\partial_L\calC_{y_+}} = 0 \Big\}.
 \end{align*}
Above, the notation $\mathbb{P}_1(R)$ for a set $R$ is the space of ($\dim R$)-variate polynomials of total degree at most 1. When $K$ is a simplex we let $\mathcal{P}_1(K) = \mathbb{P}_1(K)$. On the other hand, when $K$ is a cube we let $\mathcal{P}_1(K) = \mathbb{Q}_1(K)$, the set of polynomials of degree at most 1 in any variable. In our numerical examples we will use simplices. The norm on $X_h$ is defined as
\begin{align}\label{eq:Xh-norm}
  \|\cdot\|_{X_h} = \|\nabla \cdot\|_{L^2(y^a,\calC_{y_+})} . 
\end{align}

The algorithm in \cite{RHNochetto_EOtarola_AJSalgado_2014a} now prescribes a Galerkin approximation to \eqref{eq:parabilinear_s} from the subspace $X_h$. I.e., the solution $\mathcal{V}^{\calN} \in X_h$ is defined via the condition
\begin{align}\label{eq:VN-solution}
  a(\mathcal{V}^\calN, \phi; \mu) &= f(\phi; \mu), & \forall\;\;\phi &\in X_h.
\end{align}
We refer to \cite{RHNochetto_EOtarola_AJSalgado_2014a} for detailed approximation results. In particular when $\Omega$ is convex, we have
\begin{align*}
  \|\calU(\cdot,0;\mu)-\mathcal{V}^\calN(\cdot,0;\mu)\|_{\Hs} &\le c |\log(h_{\mathcal{T}_\Omega})|^s \; h_{\mathcal{T}_\Omega} \; \|f(\cdot;\mu)\|_{\mathbb{H}^{1-s}(\Omega)} 
\end{align*}
provided $y_{+} \sim \log(\#\mathcal{T}_{y_{+}})$. In both cases the constant $c > 0$ is independent of $h_{\mathcal{T}_\Omega}$.

Our RBM truth approximation is an amendment of the Galerkin prescription above, by replacing $a$ with $a^{Q_h}_{\mathrm{EIM}}$ that is defined in \eqref{eq:a-eim}: Find $\calU^\calN \in X_h$ such that
\begin{align}\label{eq:UN-solution}
  a_{\mathrm{EIM}}^{Q_h}(\calU^\calN, \phi; \mu) &= f(\phi;\mu), & \forall\;\;\phi &\in X_h.
\end{align}
Note that while the approximation $h^{Q_h}_{\mathrm{EIM}}$ defining $a^{Q_h}_{\mathrm{EIM}}$ is built over a $y$-truncated domain $[y_-^i, y_+]$, $i=1, 2$, the finite element approximation takes place on $[0, y_+]$, i.e., it truncates only to an upper value of $y$. Our overall goal is not to compute an approximate solution $\calU^\calN$ to the extended problem, but instead an approximate solution to the original fractional problem \eqref{eq:ppde}. Using \eqref{eq:extension-solution}, we conclude that such an approximation is provided by the trace of this solution on the bottom of the cylinder. Therefore, our final truth solution is
\begin{align}\label{eq:ucalN-def}
  u^\calN(\cdot; \mu) \coloneqq \calU^\calN(\cdot, 0; \mu).
\end{align}
Having defined the truth approximation, we now also define the reduced basis approximation $\calU_N$ introduced in \eqref{eq:rbm-def-and-certification}. With $\mu_1, \ldots, \mu_N$ chosen, then the RBM approximation is the Galerkin approximation from the span of the snapshots $\calU^\calN(\cdot;\mu_n)$: find $\calU_N \in \mathrm{span} \{\calU^\calN(\cdot;\mu_1), \ldots, \calU^\calN(\cdot;\mu_N)\}$ such that
\begin{align}\label{eq:URBM-def}
  a_{\mathrm{EIM}}^{Q_h}(\calU_N, \phi; \mu) &= f(\phi;\mu), & \forall\;\;\phi &\in \mathrm{span} \{\calU^\calN(\cdot;\mu_1), \ldots, \calU^\calN(\cdot;\mu_N)\}.
\end{align}
This condition, with the ansatz \eqref{eq:rbm-def-and-certification}, uniquely defines the coefficients $c_n$ if the snapshots are linearly independent. The RBM Galerkin approximation results in a linear system of size $N$, but requires evaluations of the bilinear form with inputs from the $\calN$-dimensional space $X_h$. Under the affine construction in the previous section, RBM algorithms circumvent $\calN$-dependent complexity in an online phase by arranging all $\calN$-dependent operations into the offline phase. That this is possible depends on the \HA{affine decomposition described} in section \ref{ssec:eim}. We refer to \cite{Rozza_Huynh_Patera, Grepl_Maday_Nguyen_Patera} for details on efficient RBM implementations.

Just as with the truth approximation, the trace of the RBM solution provides an approximation to the solution of \eqref{eq:frac_elliptic_pde}:
\begin{align}\label{eq:uN-def}
  u_N(\mu) \coloneqq \calU_N(\cdot, 0; \mu).
\end{align}

\subsection{Identification of $\mu_n$: \textit{a posteriori} error analysis}\label{sec:analysis}

The parameter values $\mu_1, \ldots \mu_N$ in practical RBM approximations are computed greedily, identified as parameter values that maximize an \textit{a posteriori} error estimate. Ideally, these parameter values are chosen via the greedy optimization
\begin{align}\label{eq:ideal-greedy}
  \mu_{n+1} = \argmax_{\mu \in D} \|e(\mu)\|_{X_h} \coloneqq \argmax_{\mu \in D} \left\| \calU^\calN(\cdot; \mu) - \calU_N(\cdot;\mu) \right\|_{X_h}. 
\end{align}
However, this explicitly requires evaluation of the truth solution $\calU^\calN$. RBM algorithms instead use a finite element \textit{a posteriori} estimate as the objective function, which approximates the error $e(\mu)$. We provide a specialized \textit{a posteriori} error estimate in this section for our problem. The main novelty of our result is that we estimate the error of the solution trace on $\Omega$ instead of in the extended cylinder $\calC_{y_+}$. The former is more natural in this setting since our ultimate goal is to approximate $u$ and not $\calU$.

Due to linearity of the fractional Laplacian and the resulting bilinear form, we have 
\begin{align}\label{eq:residual-def}
  a^{Q_h}_{\mathrm{EIM}} (e(\mu),\phi;\mu) = r(\phi;\mu) \coloneqq f(\phi) - a^{Q_h}_{\mathrm{EIM}} (\calU_N(\mu),\phi;\mu).
\end{align}
The \textit{inf-sup} constant plays a crucial role in existence and uniqueness of solutions to elliptic PDE's,
\begin{align}\label{eq:infsup-constant}
\beta_h (\mu) \coloneqq \inf_{w \in X_h}
\sup_{v \in X_h}\frac{a^{Q_h}_{\mathrm{EIM}}(w,v;\mu)}{\lVert w\lVert_{X_h}\lVert v \lVert_{X_h}}.
\end{align}
The following result provides a bound on the error $e$ on the bottom of the cylinder in terms of the residual $r$. 
\begin{theorem}\label{thm:rbm-error-estimate}
  Assume that a function $\beta_{\mathrm{LB}}(\mu)$ satisfies
  \begin{align}\label{eq:beta-LB-def}
    0 < \beta_{\mathrm{LB}}(\mu) &\leq \beta_h(\mu), & \mu &\in D.
  \end{align}
  Then
  \begin{align}\label{eq:error-bound}
    \| u^\calN(\cdot;\mu) - u_N(\cdot;\mu) \|_{\H^s(\Omega)} = \|\mathrm{tr}_\Omega e(\mu) \|_{\H^s(\Omega)} \leq \frac{\lVert r(\cdot;\mu) \rVert_{X_h '}}{ \sqrt{d_s} \beta_{\rm LB} (\mu)}
  \end{align}
  where $\mathrm{tr}_\Omega g$ for $g \in \HL(y^a,\calC)$, is the trace of $g$ on the cylinder bottom $\Omega$.
\end{theorem}
\begin{proof}
  Define $T_\mu : X_h \rightarrow X_h$ as
  \begin{align*}
    a^{Q_h}_{\mathrm{EIM}}(w, v; \mu) &= \left( T_\mu w, v \right)_{X_h}, & w, v &\in X_h.
  \end{align*}
  Therefore, 
  \begin{align*}
    \beta_{\mathrm{LB}}(\mu) \le \beta_h(\mu) = \inf_{w \in X_h} \sup_{v \in X_h} \frac{a^{Q_h}_{\mathrm{EIM}}(w, v; \mu)}{\|w\|_{X_h} \|v\|_{X_h}} = \inf_{w \in X_h} \frac{\|T_\mu w\|_{X_h}}{\|w\|_{X_h}}.
  \end{align*}
  Taking $w = e(\mu) \in X_h$ under the infimum yields
  \begin{align}\label{eq:eXh}
    \|e(\mu)\|_{X_h} \leq \frac{ \|T_\mu e(\mu)\|_{X_h}}{\beta_{\mathrm{LB}}(\mu)} = \frac{ \|r(\cdot, \mu)\|_{X_h'}}{\beta_{\mathrm{LB}}(\mu)}.
  \end{align}
  We now exercise a trace \HA{estimate, which follows from the proof of \cite[Proposition~2.1]{ACapella_JDavila_LDupaigne_YSire_2011a},} which for our setting states:
  \begin{equation}\label{eq:trace}
    \|\tr e(\mu) \|_{\H^s(\Omega)} \le d_s^{-\frac12}\| e(\mu) \|_{X_h} . 
  \end{equation}
  Combining \eqref{eq:trace} with \eqref{eq:eXh} yields the result.
\end{proof}
The utility of this result is that, while the error $e$ (and its trace) is not computable without the truth approximation $\calU^\calN$, the residual $r$ operator $r(\cdot;\mu)$ can indeed be computed via \eqref{eq:residual-def}. Additionally, our error estimate above applies only to the error on the bottom trace $\Omega$ of $\mathcal{C}$. This is a natural goal-oriented approach since our ultimate desire is not to approximate $\calU$, but instead to approximate $u$.

Such an RBM algorithm then chooses parameter snapshots $\mu_{n+1}$ not via \eqref{eq:ideal-greedy}, but instead via
\begin{align}\label{eq:greedy}
  \mu_{n+1} = \argmax_{\mu \in D} \frac{\lVert r(\cdot;\mu) \rVert_{X_h '}}{ \sqrt{d_s} \beta_{\rm LB} (\mu)}.
\end{align}
As with many RBM algorithms, the computation of $\beta_{\mathrm{LB}}(\mu)$ for all $\mu \in D$ is thus necessary. The relations \eqref{eq:beta-LB-def} and \eqref{eq:infsup-constant} indicate that this essentially requires an estimate from below of the inf-sup constant $\beta_h$. An algorithm that is frequently utilized in RBM for this purpose is the Successive Constraint Method (SCM) \cite{HuynhSCM, HKCHP}. The method resembles RBM by calculating $\beta_h$ at a few strategical locations determined by a greedy procedure, and subsequently devising an efficient linear program to construct a global function $\beta_{\mathrm{LB}}$ satisfying \eqref{eq:beta-LB-def} rigorously. However, the form of $\beta_{\mathrm{LB}}(\mu)$ in this paper is non-standard since the norm \HA{$\|\cdot\|_{X_h}$} depends on $\mu$, see \eqref{eq:Xh-norm}. We therefore need to tailor the SCM procedure accordingly, which is the topic of the next subsection.

\subsection{Fractional successive constraint method}
Our construction of $X_h$ outlined in Section \ref{ssec:fem} requires that the norm of $X_h$ depends on $\mu$ (see \eqref{eq:Xh-norm}, in particular), and in this case a standard SCM method cannot be used to efficiently compute $\beta_{\mathrm{LB}}$.

We ameliorate this situation by using an $s$-independent dominating norm. In other words, suppose we have
\begin{equation}
\label{eq:normrelation}
  \lVert \omega \lVert_{X_h} = \lVert \omega \lVert_{X_h(\mu)} \le \eta(\mu) \lVert \omega \rVert_*,
\end{equation}
for some norm $\|\cdot\|_\ast$ that does not depend on $s$. Then using the affine decomposition of $a^{Q_h}_{EIM}$ in \eqref{eq:a-eim}, we have that $\beta_h$ can be bounded by 
\begin{align}\label{eq:betaast-def}
\beta_h (\mu) = \inf_{\omega \in X_h}
\sup_{v \in X_h}\frac{a_{\mathrm{EIM}}^{Q_h}(\omega,v;\mu)}{\lVert \omega \lVert_{X_h(\mu)}\lVert v \lVert_{X_h(\mu)}} 
\ge \inf_{\omega \in X_h}
\sup_{v \in X_h}\sum_{q = 1}^{Q_h} \frac{\theta_q(\mu)}{{\eta(\mu)^2}}\frac{a^q_h(\omega,v)}{\lVert \omega \lVert_{*}\lVert v \lVert_{*}} \coloneqq \beta_{h\ast}(\mu)
\end{align}
Since the norms $\|\cdot\|_\ast$ are now $\mu$-independent, then we can apply standard SCM-based linear programs. From the inequality above for $\beta_{h\ast}$, the following result is now evident:
\begin{theorem}
  If $\beta_{\mathrm{LB}}(\mu)$ satisfies 
  \begin{align}\label{eq:betaLB-ast}
    0 < \beta_{\mathrm{LB}}(\mu) &\leq \beta_{h\ast}(\mu), & \mu&\in D,
  \end{align}
  then condition \eqref{eq:beta-LB-def} in Theorem \ref{thm:rbm-error-estimate} is satisfied.
\end{theorem}
Note that $\beta_{h\ast}(\mu)$ has affine dependence on $\mu$, so that standard SCM algorithms can be used to construct a function $\beta_{\mathrm{LB}}(\mu)$ satisfying \eqref{eq:betaLB-ast}. See \cite{HuynhSCM, CHMR-M2an} for a detailed description of these algorithms.

It remains for us to establish that it is possible to construct a relation \eqref{eq:normrelation}, that allows us to achieve the inequality \eqref{eq:betaast-def}. We provide one concrete strategy below that identifies a particularly simple form for $\beta_{h\ast}$:
\begin{proposition}
  Assume $\beta_h(\mu) > 0$ for all $\mu$. The function $\beta_{h\ast}$ of the form
  \begin{align}\label{eq:betaast-example}
    \beta_{h\ast}(\mu) &= \inf_{\omega \in X_h} \sup_{v \in X_h} \sum_{q=1}^{Q_h} \widetilde{\theta}_q(\mu) \widetilde{a}_q(w,v),
  \end{align}
  where
  \begin{align*}
    \widetilde{\theta}_q(\mu) \coloneqq \frac{\theta_q(\mu)}{\sqrt{\sum_{j=1}^{Q_h} \theta_j^2(\mu)}}, \quad \widetilde{a}_q(w,v) &\coloneqq \frac{a_q(w,v)}{\|w\|_\ast \|v\|_\ast} ,
    \quad {\|\cdot\|_* := \left( \sum_{q=1}^{Q_h} a(\cdot, \cdot; s_q)^2 \right)^{\frac12}} 
  \end{align*}
  satisfies $\beta_{h\ast}(\mu) \leq \beta_h(\mu)$.
\end{proposition}
\begin{proof}
  Define the function $\eta$ as
  \begin{align*}
    \eta(\mu)^2 &\coloneqq \left(\sum_{q=1}^{Q_h} \theta^2_q(\mu)\right)^{\frac12}.
  \end{align*}
  Then we have
  \begin{align*}
    \|w\|^2_{X_h(\mu)} &= \aEIM (w, w; \mu)
                =  \sum_{q=1}^{Q_h} \theta_q(s) a_q(w, w)
               \le \left( \sum_{q=1}^{Q_h} \theta_q(s)^2 \right)^{\frac12}
                    \left( \sum_{q=1}^{Q_h} a_q(w, w)^2 \right)^{\frac12} \\
               &= \left( \sum_{q=1}^{Q_h} \theta_q(s)^2 \right)^{\frac12}
                    \left( \sum_{q=1}^{Q_h} a(w, w; s_q)^2 \right)^{\frac12}   
                = {\eta(\mu)^2} \|w\|^2_*
  \end{align*}
  This achieves \eqref{eq:normrelation}, so that the inequality \eqref{eq:betaast-def} holds. Using the expressions for $\eta$ and $\|\cdot\|_\ast$ in the definition of $\beta_{h\ast}$ yields the relation \eqref{eq:betaast-example}.
\end{proof}
\begin{remark}
  The assumption that $\beta_h(\mu) > 0$ for all relevant $\mu = (s,\nu)$ is necessary so that $\|\cdot\|_{X_h(\mu)}$ defines a proper norm. For our problem, we restrict values of $s$ to an interval $[\epsilon, 1 - \epsilon]$ for some $\epsilon > 0$ in practice so that the positivity assumption on $\beta_h$ is reasonable.
\end{remark}
Note that the lemma above provides a generic strategy for bounding a $\beta_h(\mu)$ function having parameter-dependent norms by a parameter-affine function $\beta_{h\ast}$, assuming the operator $\aEIM$ has affine dependence.

\subsection{Algorithm summary}\label{ssec:alg}
We summarize our algorithm for constructing an RBM solution $u_N(\cdot;\mu)$ to \eqref{eq:ppde}. Evaluation of this solution for a given $\mu \in D$ has an asymptotic complexity of $N^3$; in practice $N \sim 10$ is sufficient to attain satisfactory accuracy, which we demonstrate in the numerical results section. Most of the steps in this section are standard in RBM algorithms.

\vskip 10pt
\noindent\textit{\underline{Offline phase}: Select $\mu_n$ and compute $\calU^\calN(\cdot;\mu_n)$ for $n = 1, \ldots, N$.}
\begin{enumerate}
  \item The function $h(y;s)$ defined in \eqref{eq:h-def} \HA{is approximated} by an affine decomposition \eqref{eq:h_eim}, which defines the operator $a_{\mathrm{EIM}}^{Q_h}$ in \eqref{eq:a-eim}. This also allows $\beta_{h\ast}$ to be constructed as in \eqref{eq:betaast-def} or \eqref{eq:betaast-example}, and subsequently enables construction of $\beta_{\mathrm{LB}}$ satisfying \eqref{eq:betaLB-ast} via SCM algorithms.
  \item $\mu_1 \in D$ is selected arbitrarily (usually at random). The solution $\calU^\calN(\cdot; \mu_1)$ to \eqref{eq:UN-solution} is computed.
  \item For $n = 1, \ldots, N-1$, the parameter value $\mu_{n+1}$ is chosen via \eqref{eq:greedy}, and $\calU^\calN(\cdot;\mu_n)$ is computed from \eqref{eq:UN-solution}.
\end{enumerate}
Remarks: 
\begin{itemize}
  \item The EIM decomposition of $h$ in \eqref{eq:h_eim} can be accomplished usually to an extremely small tolerance, see the numerical results section.
  \item Computing the maximum over $D$ in \eqref{eq:greedy} is usually \HA{done} only over a discrete mesh on $D$.
\end{itemize}

\vskip 10pt
\noindent\textit{\underline{Online phase}: given $\mu \in D$, evaluate $u_N(\cdot;\mu)$.}
\begin{enumerate}
  \item Compute $c_n(\mu)$ coefficients via the size-$N$ Galerkin system \eqref{eq:URBM-def}.
  \item Assemble $\calU_N$ via \eqref{eq:rbm-def-and-certification}.
  \item Restrict to the cylinder bottom $\Omega$ to obtain $u_N$, \eqref{eq:uN-def}.
\end{enumerate}
Remarks: 
\begin{itemize}
  \item The Galerkin system \eqref{eq:URBM-def} implemented naively requires $\calN$-dependent operations. Under the assumption of affine dependence of $a$ and $f$ on $\mu$, then this $\calN$-dependent complexity can be completely shifted to the offline phase.
  \item If only $u_N$ is desired, storage of the full solutions $\calU^\calN(\cdot;\mu_n)$ is unnecessary during the offline phase because of the restriction \eqref{eq:uN-def}. In this case, only $u^\calN(\cdot;\mu_n)$ need be stored in the offline phase.
\end{itemize}

\section{Numerical results}\label{sec:results}

In this section we demonstrate the efficacy of the RBM algorithm that we have devised for the fractional Laplace problem \eqref{eq:ppde}. Our test problem will be \eqref{eq:ppde}, either with $\mu = s$, or with $\mu = (s, \nu)$ for a one-dimensional parameter $\nu$. In all cases the spatial domain $\Omega$ is a rectangular two-dimensional set:
\begin{align*}
  \Omega &= [0,1]^2. 
\end{align*}
We use $5,000$ elements on $\Omega$ to form a triangulation $\mathcal{T}_\Omega$, and choose $M = M_{\mathrm {FE}} \coloneqq 158$ in \eqref{eq:grading} to define the partition $\mathcal{I}_{y_+}$ of $[0, y_+]$. The grading parameter $\gamma$ is chosen as $\gamma = 6$ when $s \le \frac{1}{2}$, and $\gamma = 2$ otherwise. We therefore have that the dimension of the truth approximation is $\dim X_h = \calN = 413,559$. 
Here, the values of $y_+$ and of $M$ are determined as linear functions of $\log_{10}$ of the number of elements in $\mathcal{T}_\Omega$ to control the error resulting from the truncation of the $y$-domain \cite{RHNochetto_EOtarola_AJSalgado_2014a}. In our computations, this results in $y_+ = 2.233$.

\subsection{EIM decomposition} 
The first step in our algorithm is to accomplish the decomposition \eqref{eq:h_eim}, i.e., to identify the function $h_q$ and $\theta_q$ in that equation. The grid we use to produce the EIM approximation is constructed in a fashion similar to the finite element grid in \eqref{eq:grading}. Concretely, for discretization of $[y_-, y_+]$, we choose an ``EIM grid'' that is 16 times finer than the finite element grid. I.e., with $M = M_{\mathrm{FE}}$ in \eqref{eq:grading}, we construct a EIM grid via \eqref{eq:grading} using $M = 16 M_{FE}$. Note that this fine resolution in the $y$ direction for EIM does not impact the offline or online efficiency of RBM itself. We take $y_-^1 = 0$ and $y_-^2 = 3.5 \cdot 10^{-7}$, where $y_-^2$ is the first nonzero entry of $\mathcal{I}_{y_+}$. We observe that our choice of graded mesh here produces a much more accurate EIM approximation.

In Figure \ref{fig:eimresult} we demonstrate the EIM accuracy. The lower-right figure shows the accuracy of the EIM approximation error as a function of $Q_h$. We observe exponential decay of the error as $Q_h$ increases. Based on these results, we use $Q_{h,1} = 17$ and $Q_{h,2} = 22$ for the remainder of our numerical tests. 

We empirically observe \HA{that the piecewise} treatment of $h(y; s)$ and the multiplication by $y$ for $s \in D_2$ are necessary; not using the latter approach requires approximation for a nearly-singular function $h(y;s)$ near $y = 0$ and leads to loss of numerical ellipticity of discretized operators. In addition, if we the more finely graded mesh of $D_1$ for computations on $D_2$, then the the first nonzero entry of $\mathcal{I}_{y_+}$ is $8.5 \cdot 10^{-21}$, which again leads to loss of numerical ellipticity. 

\begin{figure}
\begin{center}
    \includegraphics[width=0.49\textwidth]{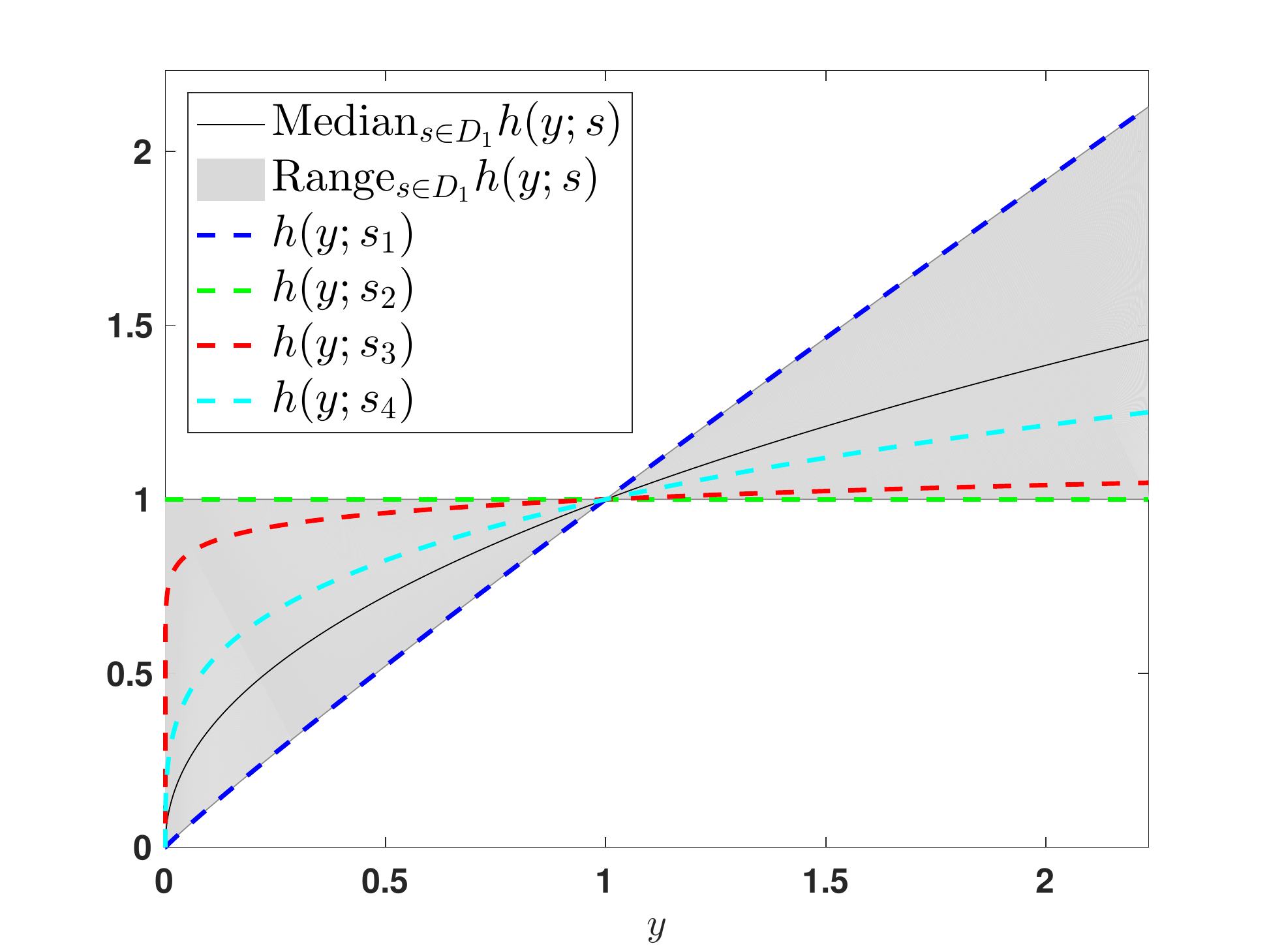}
    \includegraphics[width=0.49\textwidth]{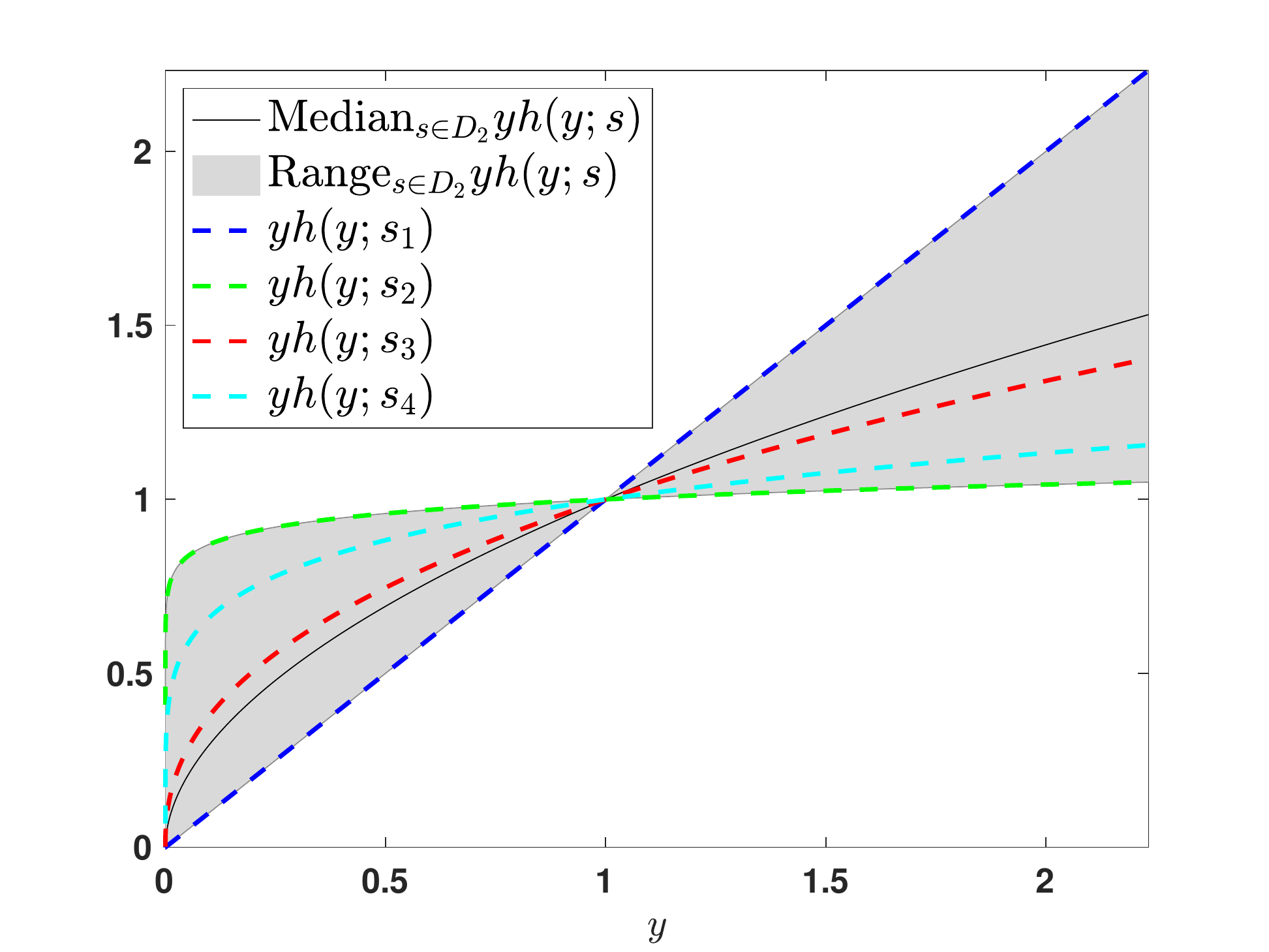}
    \includegraphics[width=0.49\textwidth]{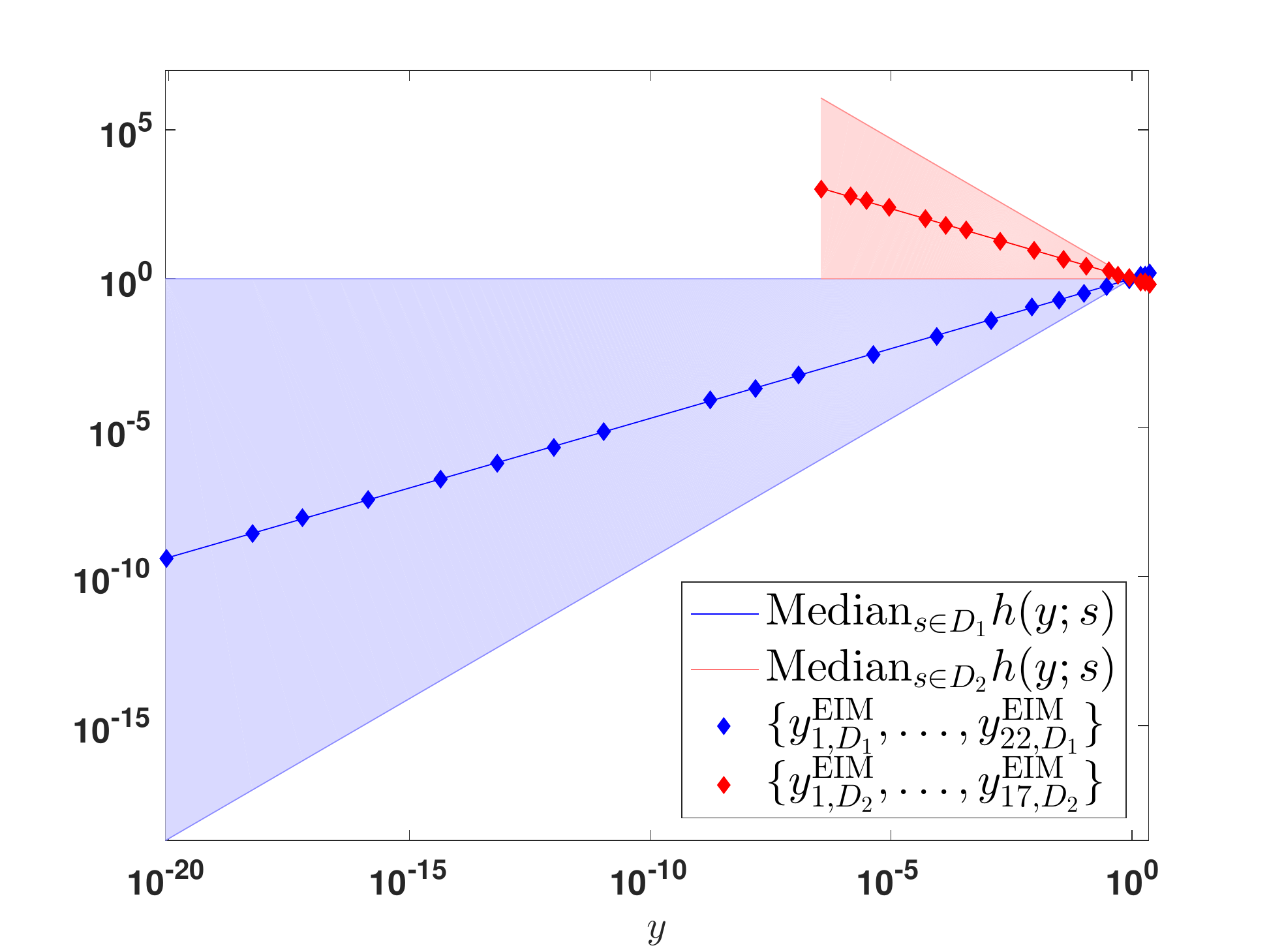}
    \includegraphics[width=0.49\textwidth]{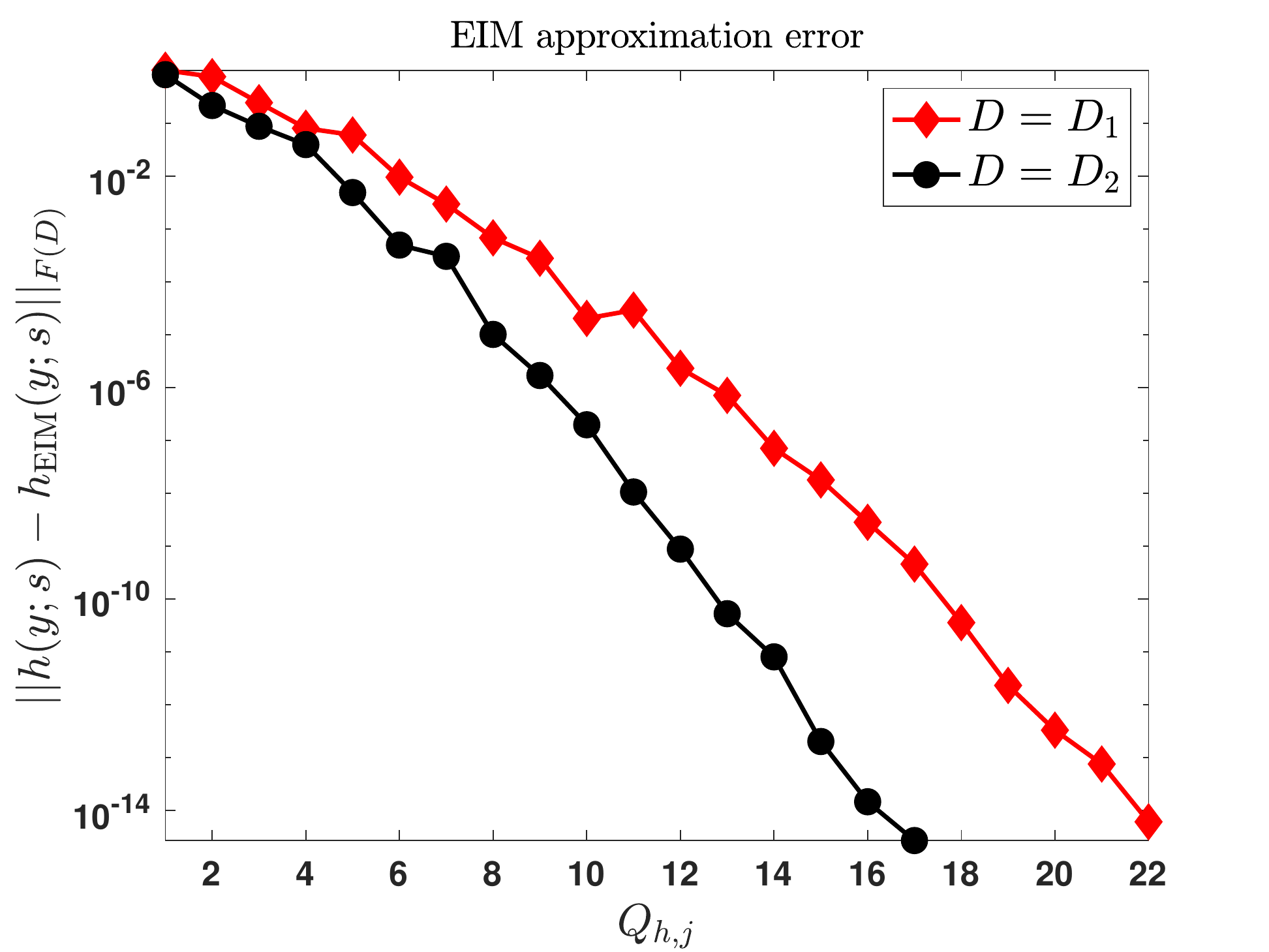}
\end{center}
\caption{EIM results for $h(y; s)$. Top: The median and range of function values $\{h(y; s): s \in D_1\}$ and $\{yh(y; s): s \in D_2\}$, and the first four EIM snapshots for each case. The shaded regions in the bottom left show the envelope of the functions $h(\cdot;s)$, and the markers show locations of the EIM interpolations points. That these interpolation points distribute evenly across the horizontal axis with a logarithmic scale demonstrate that our choice of graded mesh near $y = 0$ is necessary to obtain a good EIM approximation. Bottom right: EIM approximation error for the procedure \eqref{eq:heim-subintervals}. The norm on the space defined in \eqref{eq:F-def} is evaluated as the maximum over $2528$ points on $y \in [y_-,y_+]$ from the graded mesh and $1025$ equidistant points on $s \in D_i \subset (0,1)$. The exponential decay of error suggests that the solutions $\mathcal{V}^\calN$ and $\calU^\calN$ defined in \eqref{eq:VN-solution} and \eqref{eq:UN-solution} are proximal.}
\label{fig:eimresult}
\end{figure}

\subsection{Fractional Laplace problem}
\label{ssec:test1}

In this section we investigate model order reduction with the reduced basis method for the original problem \eqref{eq:frac_elliptic_pde}. This is equivalent to solving \eqref{eq:ppde} with $\nu$ having a fixed value. Thus, here we construct an RBM surrogate that allows quick evaluation of the map $s \mapsto u(\cdot;s)$. 
To compute parameter snapshots, we adopt the recently proposed residual-free approach \cite{JiangChenNarayan2018} to choose the RB parameter snapshots $\{\mu_1, \dots, \mu_N\}$. 
We adopt this approach instead of \eqref{eq:greedy} for simplicity; numerical results in \cite{JiangChenNarayan2018} show that the residual-free approach results in RBM approximation errors on par with residual-based approaches. Thus, instead of using \eqref{eq:greedy}, for our examples we choose $\mu_{n+1}$ via
\begin{align}\label{eq:greedy-rfree}
\mu_{n+1} = \argmax_{\mu \in D} \lVert \vec{c}(\mu)\rVert_{1},
\end{align}
i.e. the maximizer of the $L_1$-norm of the (vector of) RB Lagrange coefficients $c_n(\mu)$, where these coefficients are defined in \eqref{eq:rbm-def-and-certification}. We set
\begin{subequations}\label{eq:example1}
\begin{align}
  f(x, \nu) &= \sin(2 \pi x_1) \sin (2 \pi x_2), & x = (x_1, x_2) &\in \Omega,
\end{align}
i.e., $f$ is $\nu$-independent and we have $\mu = s$. The maximization \eqref{eq:greedy} is accomplished by replacing $D$ with a finite set. Specifically we set
\begin{align}
  D = D_1 \bigcup D_2 = [0.03, 0.5] \bigcup [0.5, 0.97],
\end{align}
\end{subequations}
and we discretize each set $D_j$, $j=1,2$, with $1025$ equi-spaced points. We collect the following data in order to plot an error metric for the RBM solution $\calU_N$:
\begin{subequations}\label{eq:error-metrics}
\begin{align}
    \mathcal{E}_N(D_j) &\coloneqq \left\{ e(z_i) \right\}_{i=1}^P = \left\{ \left\|  u^\calN(\cdot; z_i) - u_N(\cdot;z_i) \right\|_{L_2(\Omega)} \right\}_{i=1}^P, & j &= 1,2
\end{align}
with $\{z_i\}_{i=1}^P \subset D_j$ chosen as $P = 312$  equi-spaced points on $D_j$ with the first and last points removed (so that the set does not overlap with the set of $1,025$ training points for RBM). The error median, max, and minimum of the ensemble $\mathcal{E}_N$ is plotted in Figure \ref{fig:rbconv1d} as a function of $N$. Note that this error compares against the solution $\calU^\calN$ that is the solution of the affine-transformed problem \eqref{eq:UN-solution}. However, the original problem is the non-affine solution $\mathcal{V}^\calN$ defined in \eqref{eq:VN-solution}. Therefore, we will also compile the ensemble,
\begin{align}
  \mathcal{F}_N(D_j) &\coloneqq \left\{ e(z_i) \right\}_{i=1}^P = \left\{ \left\| v^\calN(\cdot; z_i) - u_N(\cdot;z_i) \right\|_{L_2(\Omega)} \right\}_{i=1}^P,& j &= 1,2, 
\end{align}
\end{subequations}
and report them in the same Figure \ref{fig:rbconv1d}.

\begin{figure}
\begin{center}
    \includegraphics[width=0.49\textwidth]{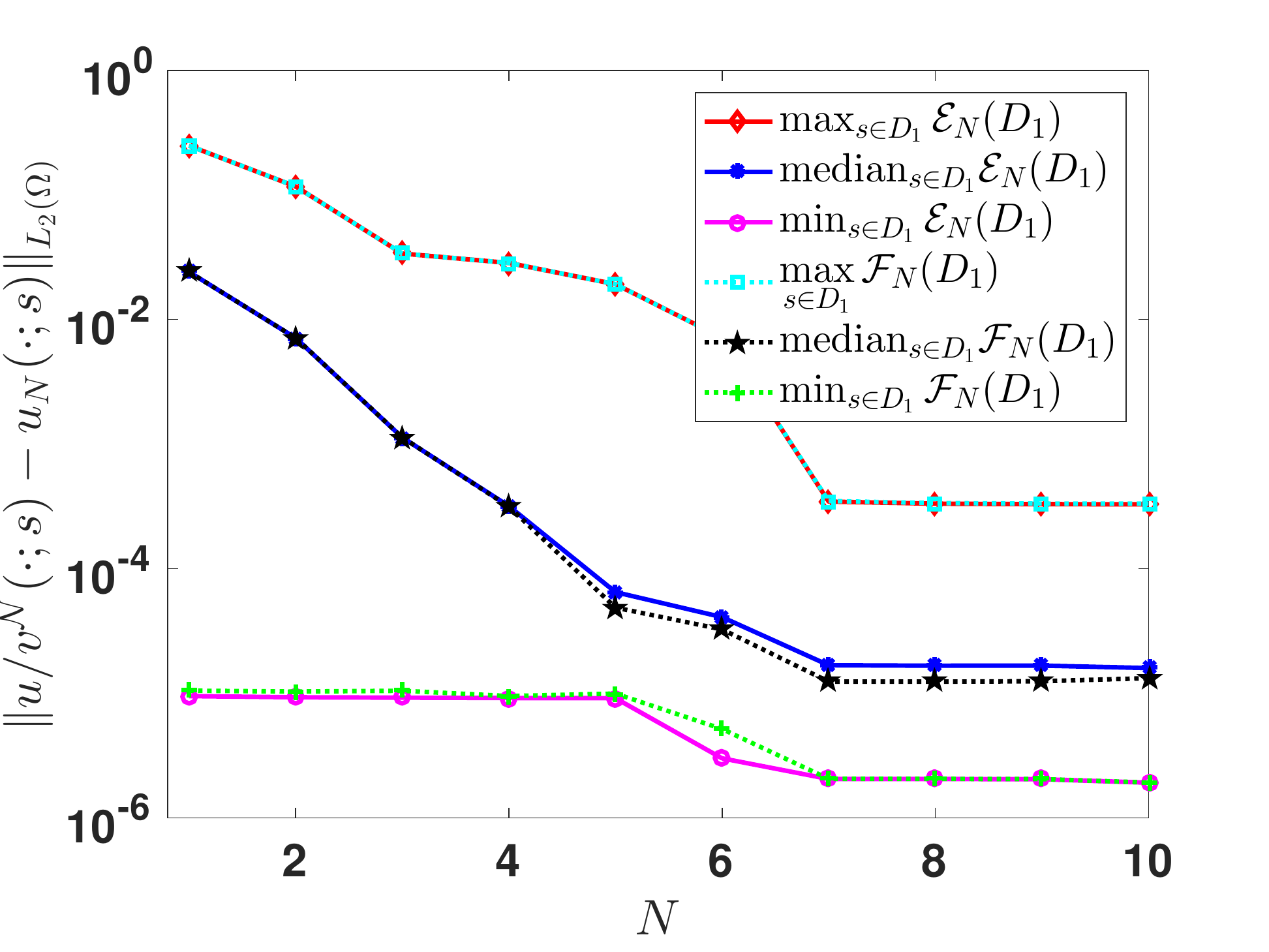}
    \includegraphics[width=0.49\textwidth]{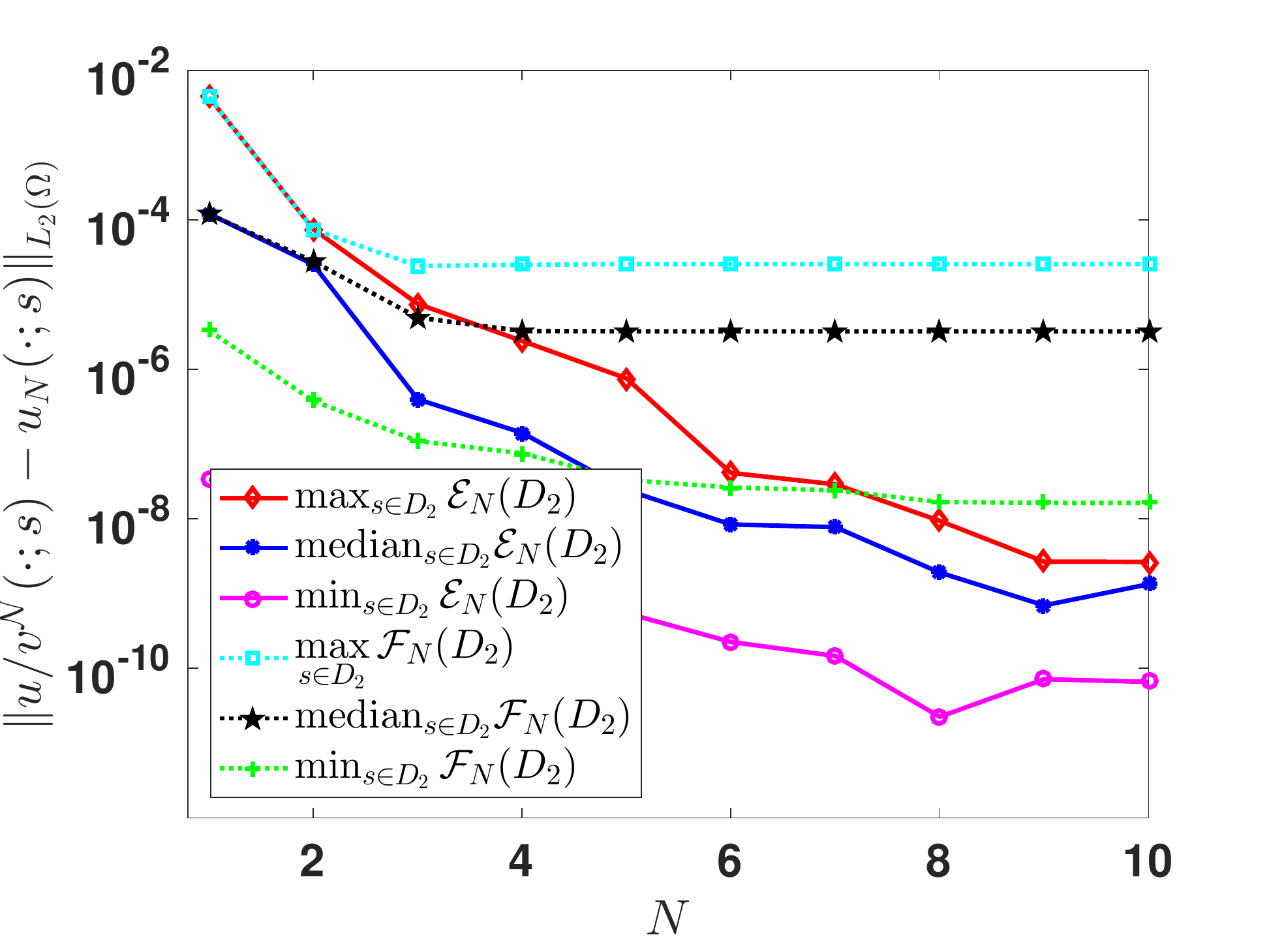}
\end{center}
\caption{Convergence of the RBM solution $\calU_N$ associated to problem \eqref{eq:example1}, where $\mu = s$ is a scalar. }
\label{fig:rbconv1d}
\end{figure}

To examine the efficiency of our scheme, we plot in Figure \ref{fig:timecomparison} the cumulative computation time for a total of $M$ queries of $u^\calN$ or $u_N$. For the $u^\calN$ this entails $M$ constructions of the finite element solution $\calU^\calN$ and subsequent restriction to the cylinder bottom. For $u_N$ this entails all computations for the one-time offline phase along with $M$ queries during the online phase and restriction to the cylinder bottom. We tabulate total cost for queries on $s \in D_1$ and $s \in D_2$ separately, and observe savings of well over 2 orders of magnitude at approximately $312$ queries. The speedup for marginal computing time (i.e., only online time) is well over four orders of magnitude when $10$ reduced basis functions are used. We also see that the RB offline time is negligible and it quickly becomes worthwhile to invest in the offline stage in the many-query setting.

Finally, we plot in Figure \ref{fig:Err_s} the RB errors as a function of parameter for $N = 2$ or $N = 7$ are used for $D_1$, and when $N = 1$ or $N = 3$ basis elements are used for $D_2$. We also mark the selected parameter values $\mu_j$. We see that the method is effective in producing an accurate surrogates over the whole domain using a very limited number of basis elements. 

\begin{figure}
\begin{center}
    \includegraphics[width=0.49\textwidth]{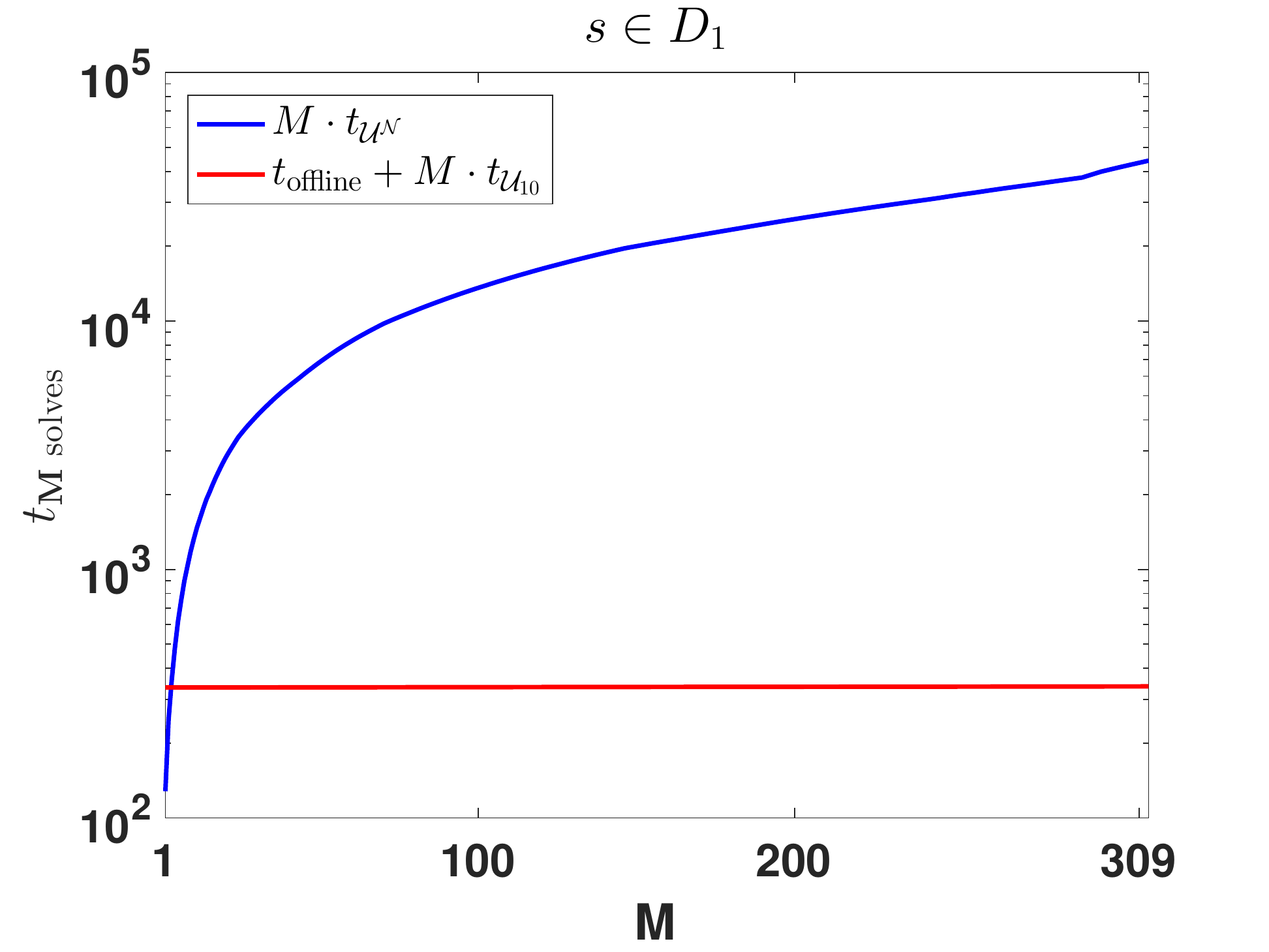}
    \includegraphics[width=0.49\textwidth]{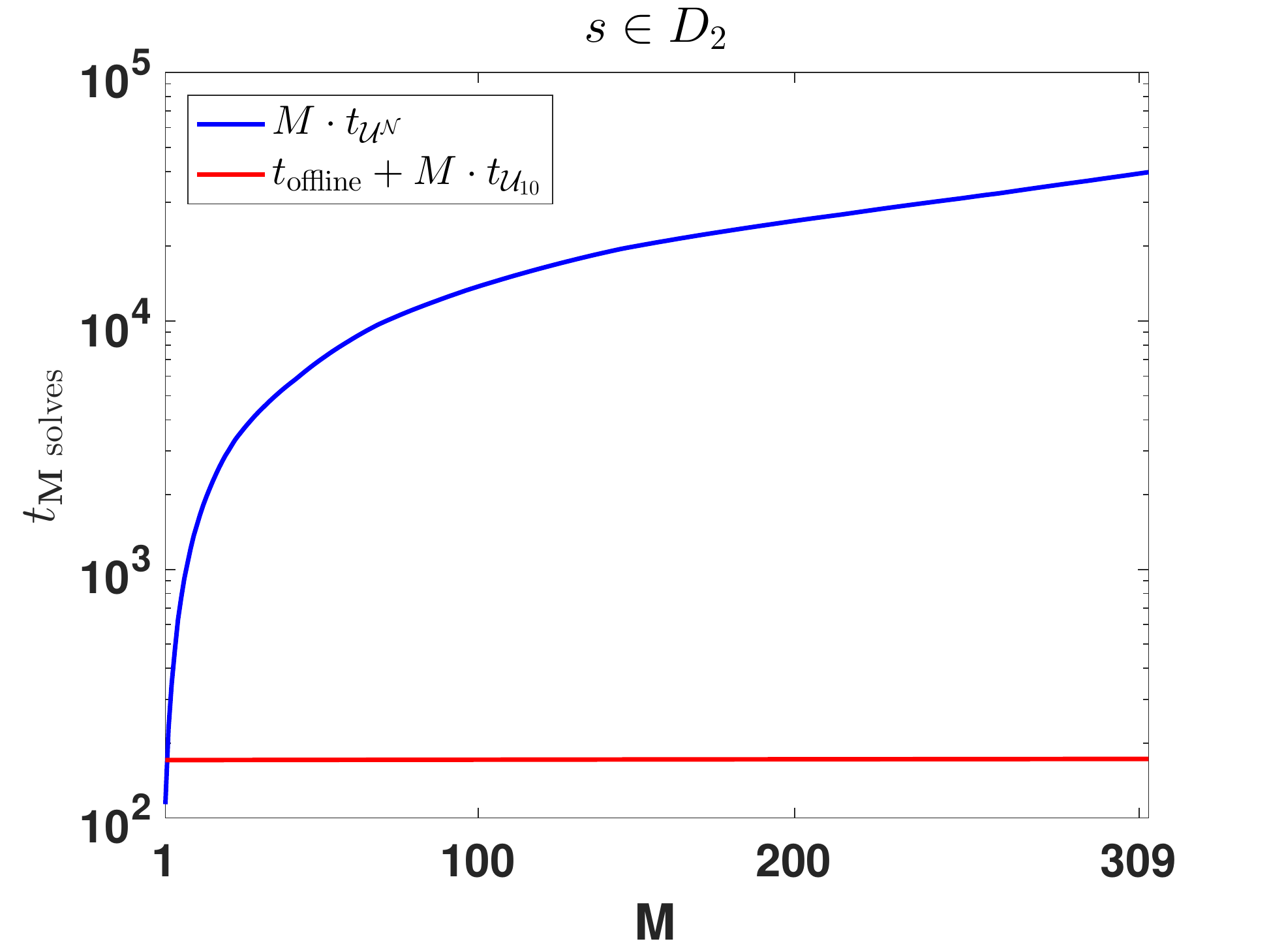}
\end{center}
\caption{The cumulative computation time for $M$ queries of the full order model $u^\calN$ and the RBM surrogate $u_N$. On the left is for the case $s \in D_1$ with $N = 7$; on the right, $s \in D_2$ with $N = 3$.}
\label{fig:timecomparison}
\end{figure}

\begin{figure}
\begin{center}
    \includegraphics[width=0.49\textwidth]{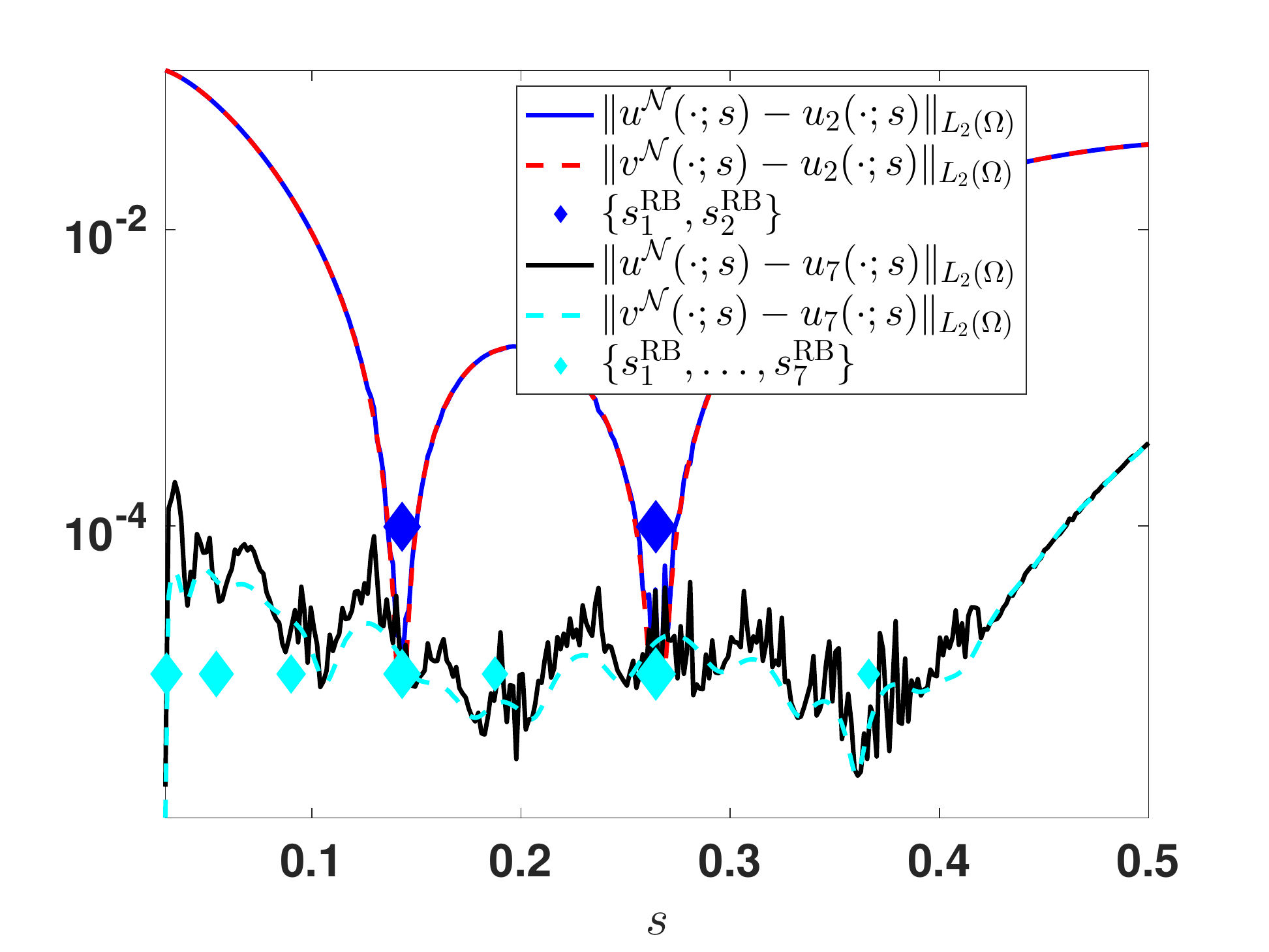}
    \includegraphics[width=0.49\textwidth]{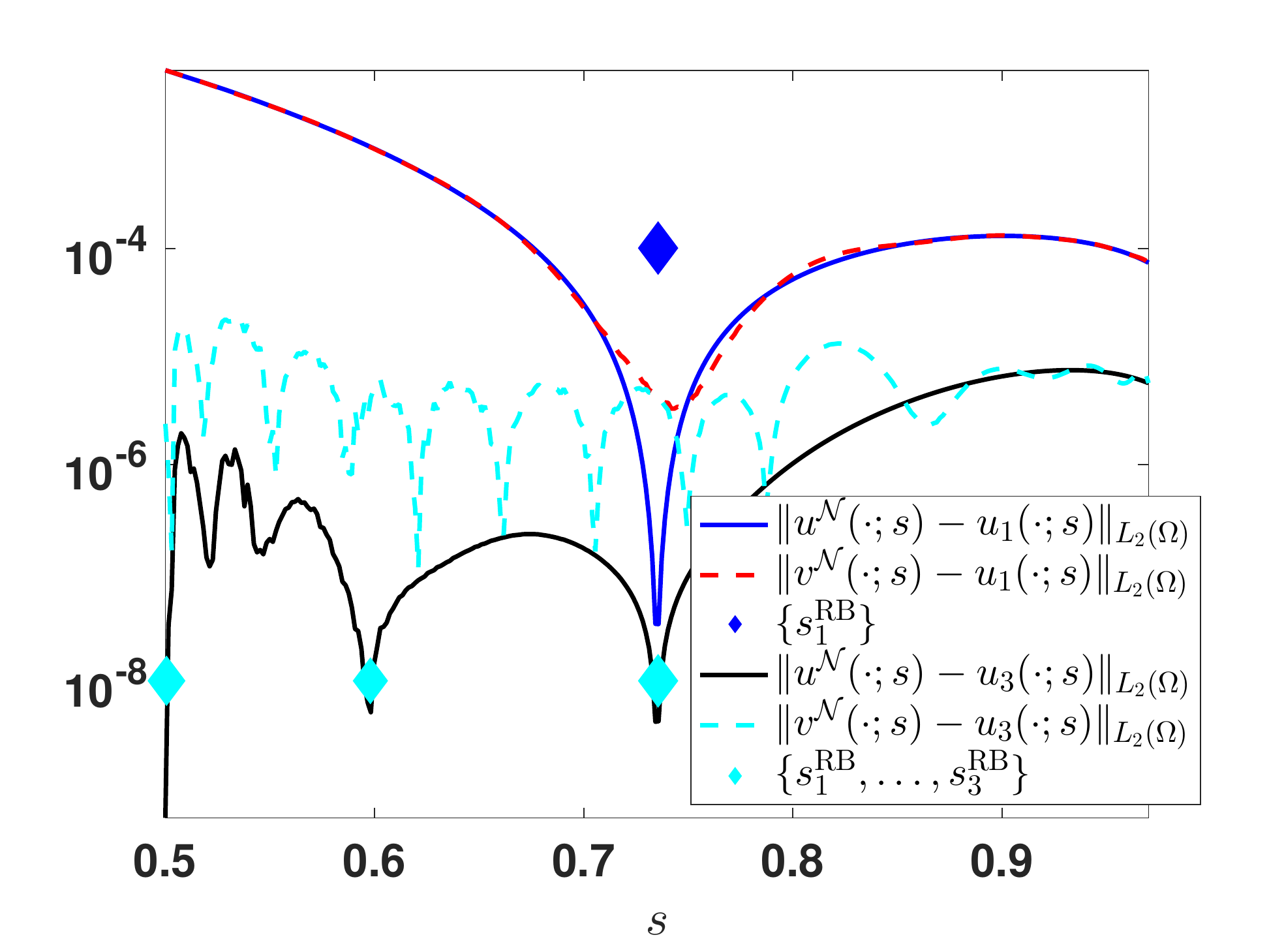}
\end{center}
\caption{RBM errors as a function of $s$ for different values of $N$.}
\label{fig:Err_s}
\end{figure}
We do observe in Figure \ref{fig:rbconv1d} that the error between the RBM surrogate $u_N$ and the restriction of $\calV^\calN$ to the cylinder base stagnates for $s \in D_2$. This is likely because we do not have $y$-uniform convergence on $[0, \infty)$ of our EIM approximation, and thus our simplified truth approximation $\calU^\calN$ retains a small discrepancy from the sought solution $\calV^\calN$. However, this stagnation occurs at relatively small values of the error, so that the RBM surrogate is still quite robust and efficient.

\subsection{Parameterized fractional Laplace problem}
For this test we take a two-dimensional parameter $\mu$:
\begin{subequations}\label{eq:example2}
\begin{align}
  \mu = (s, \nu) \in [0.03, 0.97] \times [0, 1],
\end{align}
with
\begin{align}
  f(x; \nu) = f_1(x) \nu^2 + f_2(x) (1 - \nu^2),
\end{align}
\end{subequations}
where the functions $f_1$ and $f_2$ are defined as,
\begin{align*}
  f_1(x) & = \sin(2 \pi x_1) \sin(2 \pi x_2), & f_2(x) & = \sin(3 \pi x_1) \sin(3 \pi x_2) e^{x_1 x_2}
\end{align*}
For the discrete set over which the optimization \eqref{eq:greedy-rfree} is performed, we choose a tensor-product grid, where grid in the $s$ variable is 257 equispaced points on each of $D_1$ and $D_2$, and the grid in the $\nu$ variable is 257 equispaced points on $[0,1]$. This results in a total of $66,049$ training points for the offline RBM procedure.

We again use the ensembles defined in \eqref{eq:error-metrics} to ascertain error of the RBM surrogate $u_N$, and choose the ensemble $\{z_j \}_{j=1}^P$ as $P = 900$ points, constructed as a tensorial grid with 30 points each dimension on the parameter domain. We remove points overlapping with the training set to compute errors. Error metrics for $u_N$ are shown in Figure \ref{fig:rbconv2d}. We observe similar behavior as in the previous section.

\begin{figure}
\begin{center}
    \includegraphics[width=0.49\textwidth]{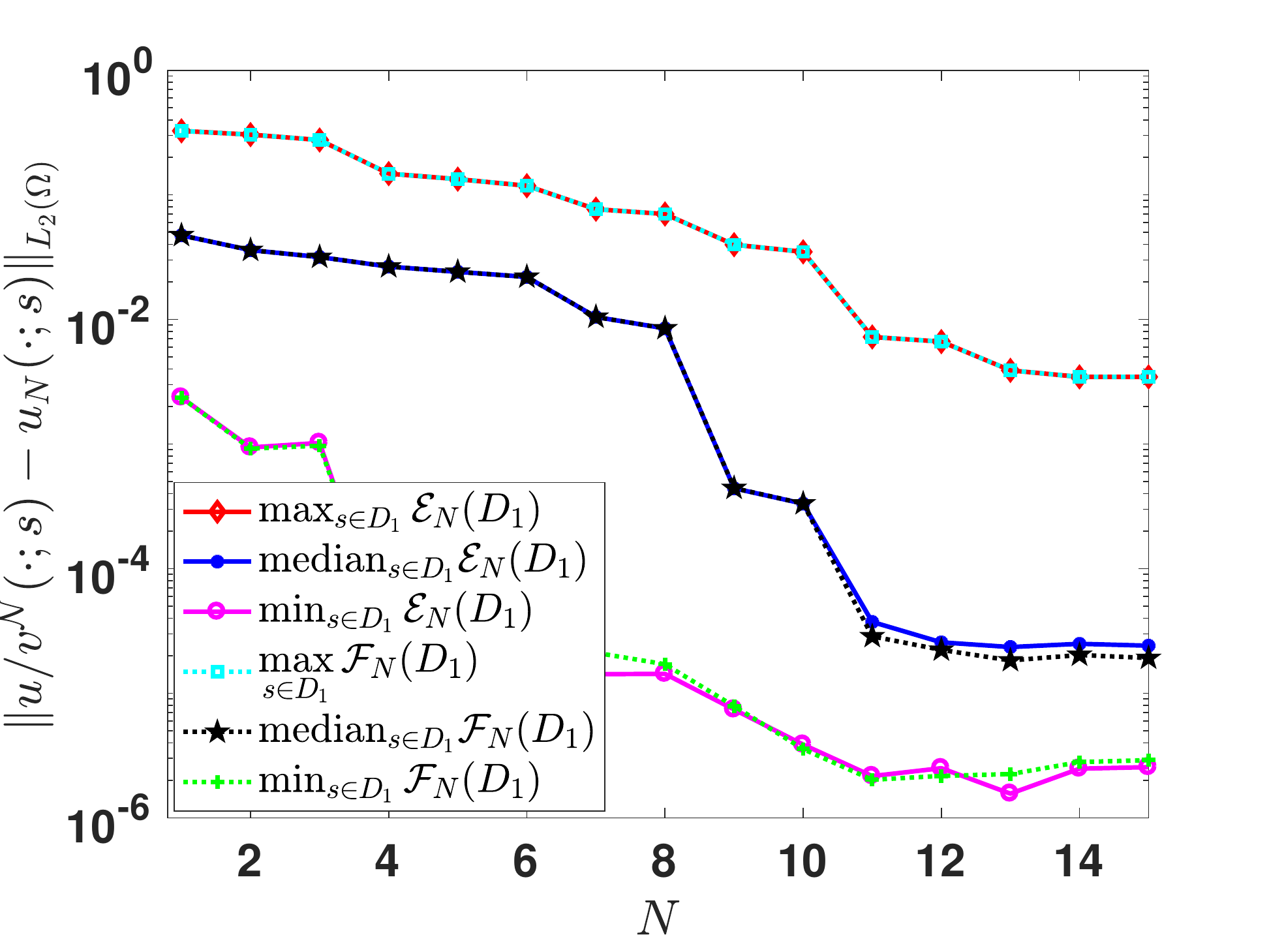}
    \includegraphics[width=0.49\textwidth]{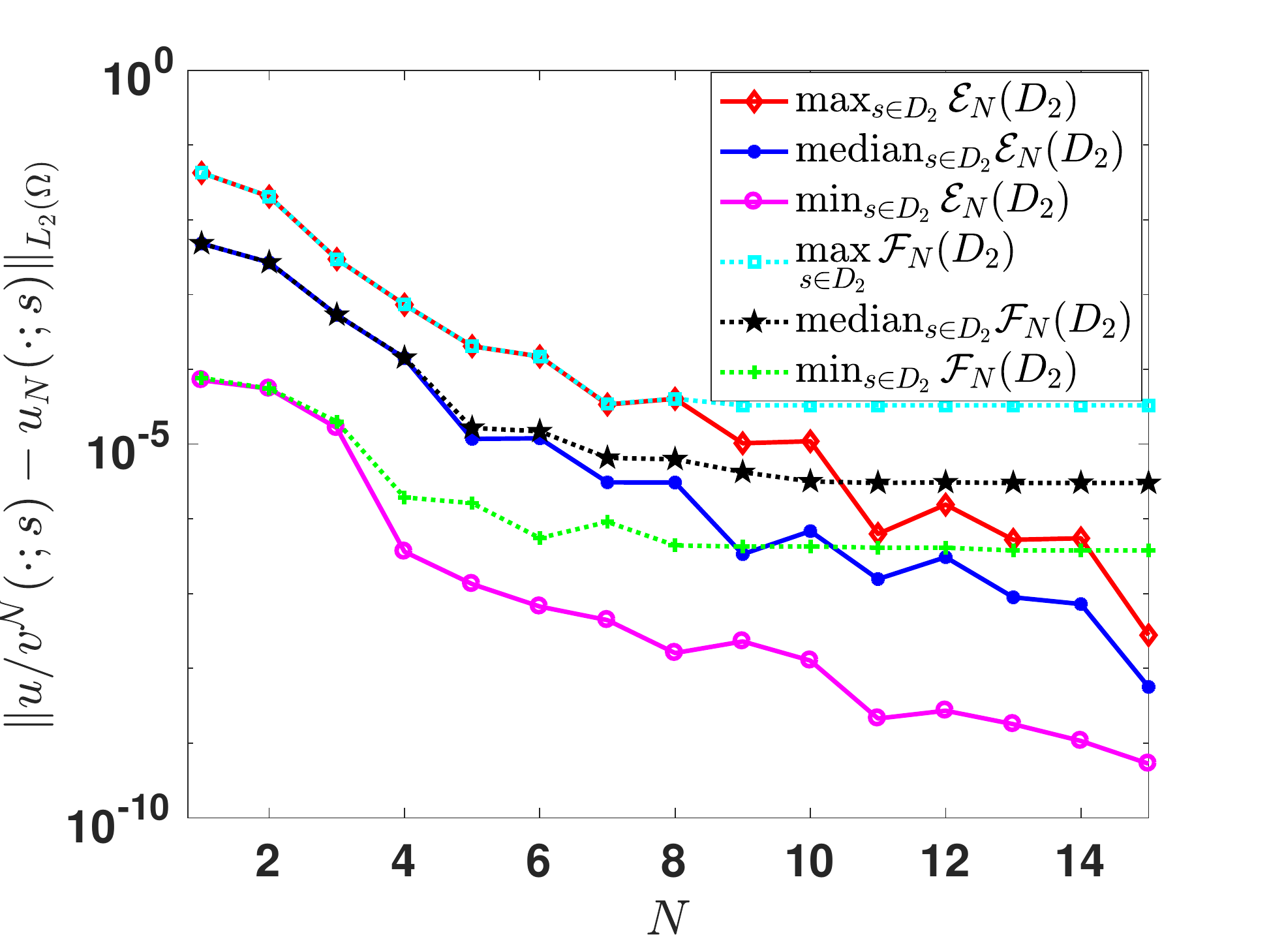}
\end{center}
\caption{Convergence of the RBM solution $\calU_N$ associated to problem \eqref{eq:example2}, where $\mu = (s,\nu)$.}
\label{fig:rbconv2d}
\end{figure}

\section{Conclusion}
We have provided a model-order reduction approach that enables fast evaluation of solutions to parameterized fractional Laplace PDEs where at least one of the parameters is the fractional exponent. Our strategy uses the reduced basis method to construct an efficient and accurate surrogate. RBM procedures do not apply ``out of the box'' to the fractional Laplace problem; in this paper we devised RBM algorithms that can address non-affine dependence of the fractional exponent in the PDE operator, as well as parameter-dependent norms in operator inf-sup conditions. Our numerical results demonstrate several orders of magnitude computational speedup over a standard finite element solver.

Our truth approximation for the RBM solver is a finite element method based on an extension solution to the fractional Laplace problem. The extension approach is seemingly cumbersome in contrast to a Dunford-Taylor solution approach, but has the significant advantage that it can be augmented to address very general nonlocal elliptic problems. We therefore expect this approach to be a cornerstone for future investigations in model order reduction for nonlocal problems.

\ifdefined\ARXIV
  \bibliographystyle{amsplain}
\else
  \bibliographystyle{siamplain}
\fi

\bibliography{akil,biblio,YChen_1,YChen_2}


\end{document}